\documentclass[a4paper,10pt]{article}
\usepackage{amsmath}
\usepackage{amssymb}
\usepackage{amscd}
\usepackage{amsthm}
\usepackage{graphicx}
\newtheorem{theorem}{Theorem}
\newtheorem{lemma}[theorem]{Lemma}
\newtheorem{corollary}[theorem]{Corollary}
\newtheorem{definition}[theorem]{Definition}
\newtheorem{result}[theorem]{Result}
\newtheorem{remark}[theorem]{Remark}
\newtheorem{example}[theorem]{Example}
\def\Q{\mathcal{Q}}
\def\B{\mathcal{B}}
\def\S{\mathcal{S}}
\def\R{\mathcal{R}}
\def\P{\mathcal{P}}
\def\PG{\mathrm{PG}}
\def\AG{\mathrm{AG}}
\def\C{\mathrm{C}(\Q^+(5,q))^{\bot}}
\def\Ha{\mathrm{C}(\mathcal{H}(5,q^2))^{\bot}}
\newtheorem{proposition}[theorem]{Proposition}
\newcommand{\q}{\mathcal{Q}(4,q)}
\title{On codewords in the dual code of classical generalised quadrangles and classical polar spaces}

\author{Valentina Pepe \thanks {The research of this author was supported by the Scientific Research Network FWO-Flanders: Fundamental Methods and Techniques in Mathematics. } \and  Leo Storme \and Geertrui Van de Voorde \thanks{This author is supported by the Fund for Scientific Research Ð Flanders (FWO - Vlaanderen).}}
\date{}
\begin{document}
\maketitle
\begin{abstract} 
 In \cite{Storme}, the codewords of small weight in the dual code of the code of points and lines of $\q$ are characterised. Inspired by this result, using geometrical arguments, we characterise the codewords of small weight in the dual code of the code of points and generators of $\mathcal{Q}^+(5,q)$ and $\mathcal{H}(5,q^2)$, and we present lower bounds on the weight of the codewords in the dual of the code of points and $k$-spaces of 
 the classical polar spaces.
 Furthermore, we investigate the codewords with the largest weights in these codes, where for $q$ even and $k$ sufficiently small, we determine the maximum weight and characterise the codewords of maximum weight. Moreover, we show that there exists an interval such that for every even number $w$ in this interval, there is a codeword in the dual code of $\mathcal{Q}^+(5,q)$, $q$ even, with weight $w$ and we show that there is an empty interval in the weight distribution of the dual of the code of $\q$, $q$ even. To prove this, we show that a blocking set of $\q$, $q$ even, of size $q^{2}+1+r$, where $0<r< (q+4)/6$, contains an ovoid of $\q$, improving on \cite[Theorem 9]{ESSS}.

\end{abstract}

\section{Introduction}

A binary $LDPC$ code $C$, in its broader sense, is a linear block code defined by a sparse parity check matrix $H$, i.e., the number of 1s in $H$ is small compared to the number of 0s in $H$. In 2001, Kou et al. \cite{11} examined classes of LDPC codes defined by incidence structures in
finite geometries. Since then, other LDPC codes have been produced based on various incidence structures in discrete mathematics and finite geometry. In
particular, Vontobel and Tanner \cite{T3} considered the LDPC codes generated by generalised polygons, focusing on generalised quadrangles. They demonstrated that some generalised quadrangle LDPC
codes perform well under the sum product algorithm \cite{16}. Later, Liu and Pados \cite{Liu} showed that
all LDPC codes derived from finite classical generalised quadrangles are quasi-cyclic, and they
gave the explicit size of the circulant blocks in the parity check matrix. Their simulation results
show that several generalised polygon LDPC codes have a powerful bit-error-rate performance when
decoding is carried out via low-complexity variants of belief propagation \cite{Liu}. 

In \cite{Storme}, the codewords of small weight in the dual code of the code of points and lines of classical generalised quadrangles are characterised, and in \cite{val}, the authors characterise codewords of small weight in the dual code of non-classical generalised quadrangles arising from the linear representation of geometries. 

In this paper, we continue this research by studying the codes and the dual codes (i.e. LDPC-codes) arising from classical polar spaces.
\subsection{Definitions and notation}
In this article, $\PG(d,q)$ denotes the projective space of dimension $d$ over the finite field $\mathbb{F}_q$, $q=p^h$, $h\geq 1$, $p$ prime, and $\theta_d$ denotes the number of points in $\PG(d,q)$, i.e. $\theta_d=\frac{q^{d+1}-1}{q-1}.$

Denote by $\mathcal{Q}^+(2n+1,q)$, $\mathcal{Q}(2n,q)$ and $\mathcal{Q}^-(2n+1,q)$, the non-singular hyperbolic quadric of $\PG(2n+1,q)$, respectively the non-singular parabolic quadric of $\PG(2n,q)$, respectively the non-singular elliptic quadric of $\PG(2n+1,q)$. The subspaces of maximal dimension on a quadric are called the {\em generators} of the quadric. The dimension $g$ of a generator is $n-1$ for $\mathcal{Q}^-(2n+1,q)$ and $\mathcal{Q}(2n,q)$, and $n$ for $\mathcal{Q}^+(2n+1,q)$.
Denote by $\mathcal{H}(d,q^2)$ the Hermitian variety defined by a non-degenerate Hermitian form in the projective space $\PG(d,q^2)$. The generators of $\mathcal{H}(d,q^2)$ have dimension $\lfloor (d-1)/2 \rfloor$, where $\lfloor x \rfloor$ denotes the largest integer smaller than or equal to $x$. 

For $\mathcal{Q}^-(2n+1,q)$, $\mathcal{Q}(2n,q)$, $q$ odd, $\mathcal{Q}^+(2n+1,q)$ or $\mathcal{H}(d,q^2)$, let $\sigma$ denote the related polarity, and let $\pi^{\sigma}$ denote the image  of a subspace $\pi$ under $\sigma$. A subspace entirely contained in $\mathcal{Q}^-(2n+1,q)$, $\mathcal{Q}(2n,q)$, $q$ odd, $\mathcal{Q}^+(2n+1,q)$ or $\mathcal{H}(d,q^2)$ is called {\em self-polar} (with respect to $\sigma$).
\subsubsection{The code of points and $k$-spaces in a polar or projective space}

We define the incidence matrix $A = (a_{ij})$ 
of points and $k$-spaces of the polar space $\mathcal{P}$ in a projective space defined over a finite field of characteristic $p$ as the matrix whose rows are indexed by the $k$-spaces contained in $\mathcal{P}$ and whose columns are indexed
by the points of $\mathcal{P}$, and 
with entry
$$ 
a_{ij} = \left\{
\begin{array}{ll}
1 & \textrm{if point $j$ belongs to $k$-space $i$,}\\
0 & \textrm{otherwise.}
\end{array} 
\right.
$$
The $p$-ary linear code $\mathrm{C}_k(\P)$, $k\leq g$, with $g$ the dimension of the generators of $\P$, of points and $k$-spaces of $\P$ is 
the $\mathbb{F}_p$-span of the rows of the  incidence matrix $A$.  If we consider the code $\mathrm{C}_g(\P)$, we simply denote this by $\mathrm{C}(\P)$. The {\em support} of a codeword $c$, denoted by $supp(c)$, is the set of all non-zero positions of $c$. We denote the set of points corresponding to $supp(c)$ by $\S$, and the complement of $\S$ in $\P$ by $\B$.
 The {\em weight} of $c$ is the number of non-zero positions of $c$ and is denoted by $wt(c)$. We let $c_P$ denote the symbol of the codeword $c$ in the coordinate position corresponding to the point $P$, and let $(c_1,c_2)$ denote the inner product in $\mathbb{F}_p$ of two codewords $c_1, c_2$ of $\mathrm{C}_k(\P)$. 

The dual code $\mathrm{C}_k(\P)^\bot$ is the set of all vectors orthogonal to all codewords of $\mathrm{C}_k(\P)$, hence
$$\mathrm{C}_k(\P)^\bot=\{ v\in V(\vert \P \vert,p) || (v,c)= 0,\ \forall c\in \mathrm{C}_k(\P)\}.$$
This means that for all $c\in \mathrm{C}_k(\P)^{\bot}$ and all $k$-spaces $K$ contained in $\P$, we have $(c,K)=0$.

Similarly, one can define the incidence matrix~$B = (b_{ij})$ 
of points and $k$-spaces of the projective space $\PG(n,q)$ as the matrix whose rows are indexed by the $k$-spaces of $\PG(n,q)$ and whose columns are indexed
by the points of $\PG(n,q)$, and 
with entry
$$ 
a_{ij} = \left\{
\begin{array}{ll}
1 & \textrm{if point $j$ belongs to $k$-space $i$,}\\
0 & \textrm{otherwise.}
\end{array} 
\right.
$$
The $p$-ary linear code $\mathrm{C}_k(n,q)$ of points and $k$-spaces of  $\PG(n,q)$, $q=p^h$, $p$ prime, $h\geq 1$,  is 
the $\mathbb{F}_p$-span of the rows of the  incidence matrix $B$.

 \subsubsection{Generalised quadrangles}
 A \textit{generalised quadrangle} $\Gamma$ is a set of points and lines such that:
\begin{description}
  \item[(a)] any two distinct points are on at most one line,
  \item[(b)] every line is incident with $s+1$ points and every point
  is incident with $t+1$ lines,
  \item[(c)] if a point $p$ is not incident with the line $L$, then
  there is exactly one line through $p$ intersecting $L$.
\end{description}
The generalised quadrangle $\Gamma$ is said to have {\em order} $(s,t)$ or
$s$ if $s=t$ and is denoted by $\mathrm{GQ}(s,t)$. The number of points of $\Gamma$ is $(s+1)(st+1)$ and
the number of lines is $(t+1)(st+1)$. Dualizing $\Gamma$ we get a
generalised quadrangle of order $(t,s)$. For more information on generalised quadrangles, we refer to \cite{payne}.

Let $\q$ be a non-singular parabolic quadric in the projective space
$\PG(4,q)$; the set of points and the set of lines of $\q$ form a
generalised quadrangle of order $q$ and hence $\q$ has
$(q+1)(q^{2}+1)$ points and $(q+1)(q^{2}+1)$ lines. The points of
$\PG(3,q)$ and the self-polar lines of a symplectic polarity $\sigma$
form the generalised quadrangle $\mathcal{W}(q)$ of order $q$.
The following theorem describes the connection between $\q$ and $\mathcal{W}(q)$.

\begin{result}(\cite[Theorem 3.2.1]{payne})\label{payne}
\begin{enumerate}
  \item The generalised quadrangle $\q$ is isomorphic to the dual of
  $\mathcal{W}(q)$.
  \item The generalised quadrangles $\q$ and $\mathcal{W}(q)$ are self-dual if
and only if $q$ is even.
\end{enumerate}
\end{result}

A \textit{blocking set} of a generalised quadrangle $\Gamma$ is a set $B$
of points such that every line of $\Gamma$ contains at least one
point of $B$. A blocking set $B$ is called \textit{minimal} if no proper subset of $B$ is still a blocking set. A set $\mathcal{O}$ of points 
of $\Gamma$ is called an \textit{ovoid} if every line of $\Gamma=\mathrm{GQ}(s,t)$
contains exactly one point of $\mathcal{O}$, hence
$\mathcal{O}$ is a set of $st+1$ pairwise non-collinear points. It
is well-known that $\q$, for $q$ even, has an ovoid (see \cite{Thas}).
\begin{remark}
An ovoid in the projective space $\PG(3,q)$, $q>2$, is a set of ${q\sp 2} + 1$ points, no three of which are collinear. When $q$ is odd, the only ovoids are the elliptic quadrics, but the classification of ovoids for even $q$ is an open problem. There are only two known classes of ovoids: the elliptic quadrics (which exist for all prime powers $q$) and the Tits ovoids which exist only for $ q = 2\sp h$, where $ h \ge 3$ is odd. Only for $q\leq 32$, ovoids have been characterised as one of these types (see \cite{keefe}). By results of Thas \cite{ovoide} and Tits \cite{ovoidetits}, every ovoid of $\PG(3,q)$ corresponds to an ovoid of $\mathcal{W}(3,q)$ and vice versa. Since for $q$ even, $\mathcal{W}(3,q)$ is isomorphic to $\q$, the classification of ovoids in $\PG(3,q)$, $q$ even, is the same as for $\mathcal{Q}(4,q)$, $q$ even.
\end{remark}

The dual of a blocking set $B$ of a generalised quadrangle $\Gamma$ is a cover $\mathcal{C}$. A \textit{cover} $\mathcal{C}$ is a set of lines such that every point of $\Gamma$ lies on at least
one line of $\mathcal{C}$. If every point of $\Gamma$ lies on exactly one
line of $\mathcal{C}$, then $\mathcal{C}$ is a \textit{spread} of
$\Gamma$, i.e. $\mathcal{C}$ is a partition of the point set of $\Gamma=GQ(s,t)$ into $st+1$ pairwise non-concurrent lines. A cover $\mathcal{C}$ is {\em minimal} if there is no cover properly contained in $\mathcal{C}$. The multiplicity $\mu(P)$ of a point $P$ is the number of lines of $\mathcal{C}$ through it. The {\em excess} of a point $P$, denoted by $e(P)$, is equal to $\mu(P)-1$. A {\em multiple} point $P$ of $\mathcal{C}$ is a point with $e(P)>0$. The excess of a line $L$ is the sum of the excesses of the points on $L$.

A line $L$ of $\Gamma$ is called a \textit{good line} for $\mathcal{C}$ when $L \notin \mathcal{C}$ and $L$ does not have multiple points of $\mathcal{C}$. We
have the following result for a cover of $\q$.

\begin{lemma}(\cite[Lemma 2]{ESSS})\label{lemma1}
A cover $\mathcal{C}$ of $\q$ of size $q^{2}+1+r$, $0\leq r\leq q$,
always has a good line.
\end{lemma}

\subsubsection{Minihypers}
\begin{definition} A {\em weighted  $\{ f,m;N,q\}$-{minihyper}}
is a pair $(F,w)$, where $F$ is a subset of the point set of $\PG(N,q)$ and
$w$ is a weight function $w:\PG(N,q)\rightarrow \mathbb{N}: x \mapsto
w(x)$ satisfying
\begin{enumerate}
  \item $w(x)>0 \Longleftrightarrow x \in F$,
  \item $\sum_{x \in F}w(x)=\textit{f}$,
  \item $\min \{\sum_{x \in H}w(x)\vert \vert H \text{ is a hyperplane of }\PG(N,q)\}=m.$
\end{enumerate}
\end{definition}

\begin{remark} If the weight function $w$ has values in
$\{0,1\}$, then we say that the minihyper is {\em non-weighted} and we
refer to it by $F$.\end{remark}
  \begin{lemma}(\cite[Lemma 3.2]{DBHS})\label{anja}
Let $F$ be a non-weighted $\{x(q+1),x;4,q\}$-minihyper,
$x<\frac{q}{2}$, contained in $\q$. Then $F$ is the union of $x$
pairwise disjoint lines.
\end{lemma}
In Lemma \ref{lemma2}, we will extend the result of Lemma \ref{anja} on non-weighted minihypers of $\q$ to weighted minihypers, where the union of lines is replaced by a sum of lines.
Consider $x$ lines $L_1,\ldots, L_x$, where a given line may occur more than once. The {\em sum} $L_1+\cdots+L_x$ is the weighted set $F$ of points with weight function $w$, satisfying $w(P)=j$ if $P$ belongs to $j$ lines in $L_1,\ldots,L_x$. For example, if $L_1=\ldots=L_x$, then all points of $L_1$ have weight $x$, and all other points have weight zero.

\subsection{Small weight codewords in the code of classical polar spaces}
It is clear that $\mathrm{C}_k(\P)$, with $\P$ a polar space embedded in $\PG(\nu,q)$, is a subcode of $\mathrm{C}_k(\nu,q)$. 

Hence, we can use the following result on the codewords of small weight in the code of points and $k$-spaces of $\PG(n,q)$ to obtain a result on $\mathrm{C}_k(\P)$.
\begin{result}\cite{LSV3} The minimum weight of $\mathrm{C}_k(n,q)$, $q=p^h$, $p$ prime, is equal to $\theta_k$ and a codeword of minimum weight corresponds to a scalar multiple of an incidence vector of a $k$-space. 
There are no codewords in $\mathrm{C}_k(n,q)$ with weight in $]\theta_k,(12\theta_{k}+2)/7[$ if $p=7$ and there are no codewords in $\mathrm{C}_k(n,q)$ with weight in $]\theta_k,(12\theta_{k}+6)/7[$ if $p>7$. If $q$ is prime, there are no codewords of $\mathrm{C}_k(n,q)$ with weight in $]\theta_k,2q^k[$.
\end{result}

\begin{theorem}

The minimum weight of $\mathrm{C}_k(\P)$, with $\P$ a polar space embedded in $\PG(\nu,q)$, $q=p^h$, $p$ prime, is equal to $\theta_k$ and a codeword of minimum weight corresponds to a scalar multiple of the incidence vector of a $k$-space contained in $\P$.
There are no codewords in $\mathrm{C}_k(\P)$ with weight in $]\theta_k,(12\theta_{k}+2)/7[$ if $p=7$ and there are no codewords with weight in $]\theta_k,(12\theta_{k}+6)/7[$ if $p>7$. If $q$ is prime, there are no codewords in $]\theta_k,2q^k[$.
\end{theorem}

Note that for the dual code $\mathrm{C}_k(\P)^{\bot}$, one cannot use the results for $\mathrm{C}_k(n,q)^{\bot}$, since $\mathrm{C}_k(\P)^{\bot}$ is not a subcode of $\mathrm{C}_k(n,q)^{\bot}$.

\subsection{Summary} In this paper, small  and large weight codewords in the dual code of points and $k$-spaces in the classical polar spaces are investigated. We summarise the results of this paper in the following table. If a geometric characterisation of codewords in a certain interval is obtained, the corresponding proposition is a 'theorem'.\\

$\begin{array}{lll}
\mbox{Dual code of} &   \mbox{Small weight codewords} & \mbox{Large weight codewords}\\
\hline
\Q(4,q), \mathcal{W}(q) & \cite{Storme}, q \mbox{ even} & \mbox{Theorem 16},\ q \mbox{ even} \\

\Q^+(5,q) & \mbox{Theorem } \ref{H2} &  \mbox{Theorem } \ref{constr}, q \mbox{ even}\\

\mathcal{H}(5,q^2) & \mbox{Theorem } \ref{herm}, q >893&  \\

\Q^-(5,q) & \mbox{Proposition }\ref{prop}& \\

\mathcal{H}(4,q^2) & \mbox{Proposition } \ref{prop2}&  \\
\hline
\Q^+(2n+1,q) & \mbox{Propositions } \ref{minimum} \mbox{ and }\ref{klein}&  \mbox{Theorem } \ref{eerste},k=1,\ q \mbox{ even}\\
\Q^+(2n+1,q) &  \mbox{Propositions } \ref{minimum} \mbox{ and }\ref{klein}&  \mbox{Theorem } \ref{tweede}, k=2,\ q \mbox{ even}\\
\Q^+(2n+1,q) & \mbox{Propositions } \ref{minimum} \mbox{ and }\ref{klein} &  \mbox{Theorem } \ref{derde}, 3\leq k <(n+3)/2 , q \mbox{ even}\\
\Q(2n,q)  &  \mbox{Proposition } \ref{minimum} &  \mbox{Theorem } \ref{vierde}, 1\leq k< (n+1)/2, q \mbox{ even}\\
\Q^-(2n+1,q)  &  \mbox{Proposition } \ref{minimum} &  \mbox{Theorem } \ref{vijfde}, 1\leq k<(n+1)/2 , q \mbox{ even}\\
\mathcal{H}(n,q^2)  &  \mbox{Proposition } \ref{minimum}  &  \mbox{Theorem } \ref{zesde}, 1\leq k\leq(n-3)/2, q \mbox{ even}\\
\mathcal{H}(n,q^2), n \mbox{ odd}  &  \mbox{Propositions } \ref{minimum} \mbox{ and }\ref{ext} &  \mbox{Theorem } \ref{zesde}, 1\leq k\leq(n-3)/2, q \mbox{ even}\\

\hline

\end{array}$

\section{A lower bound on the minimum weight of the dual code of $\P$}

We have the following geometrical condition:
\textit{ If $c$ is a codeword of $\mathrm{C}_{k}(\P)^{\bot}$, then every $k$-subspace of $\P$ contains zero or at least 2 points of $\S=supp(c)$}. 

If $q$ is even, we have the following geometrical properties, which will be denoted by $(*)$ and $(**)$ throughout this article.\\

\noindent
$(*)$ \textit{$c$ is a codeword of $\mathrm{C}_{k}(\P)^{\bot}$ if and only if every $k$-subspace of $\P$ contains an even number of points of $\S$}.\\
$(**)$ \textit{$\B=\P\setminus \S$ is a blocking set with respect to the $k$--spaces of
$\P$}.

\begin{proposition}\label{minimum}
Let $d$ be the minimum weight for the code $\mathrm{C}_{k}(\P)^{\bot}$.
\begin{itemize}
\item[(a)] If $\P=\mathcal{Q}^{+}(2n+1,q)$, then $$d \geq
1+\frac{q^{n}-1}{q^{k}-1}(q^{n-1}+1).$$
\item[(b)] If $\P=\mathcal{Q}(2n,q)$, then $$d \geq
1+\frac{q^{n-1}-1}{q^{k}-1}(q^{n-1}+1).$$
\item[(c)] If $\P=\mathcal{Q}^-(2n+1,q)$, then $$d \geq
1+\frac{q^{n-1}-1}{q^{k}-1}(q^{n}+1).$$
\item[(d)] If $\P=\mathcal{H}(n,q^2)$, then $$d \geq
1+\frac{(q^{n-1}-(-1)^{n-1})(q^{n-2}-(-1)^{n-2})}{q^{2k}-1}.$$
\end{itemize}
\end{proposition}
\begin{proof}
\begin{itemize}
\item[(a)] Let $c$ be a codeword of $\mathrm{C}_{k}(\mathcal{Q}^{+}(2n+1,q))^{\bot}$. If $P$ is a point of $\S$, then every $k$--space of $\mathcal{Q}^+(2n+1,q)$ through $P$ must contain at least another point of $\S$. The number of $k$--spaces of $\mathcal{Q}^+(2n+1,q)$ through $P$ is the number of $(k-1)$--spaces of $\mathcal{Q}^+(2n-1,q)$, and this number equals (see \cite[Chapter 22]{Thas})
$$M:= \prod
\limits_{i=0}^{k-1}\frac{q^{n-i}-1}{q^{i+1}-1} \cdot \prod
\limits_{i=n-k}^{n-1}(q^{i}+1)$$ 
and the number of $k$--spaces of $\mathcal{Q}^+(2n+1,q)$ through two collinear points of $\mathcal{Q}^+(2n+1,q)$ is $$N:= \prod
\limits_{i=0}^{k-2}\frac{q^{n-i-1}-1}{q^{i+1}-1} \cdot\prod
\limits_{i=n-k}^{n-2}(q^{i}+1),$$ hence $$|\S| \geq
1+\frac{M}{N}=1+\frac{q^{n}-1}{q^{k}-1}(q^{n-1}+1).$$

\item[(b)] The same reasoning as in case (a), with $M:=\displaystyle \prod
\limits_{i=0}^{k-1}\frac{q^{n-1-i}-1}{q^{i+1}-1}\cdot \displaystyle \prod
\limits_{i=n-k}^{n-1}(q^{i}+1)$ and $N:=\displaystyle \prod
\limits_{i=0}^{k-2}\frac{q^{n-2-i}-1}{q^{i+1}-1} \displaystyle \cdot\prod
\limits_{i=n-k}^{n-2}(q^{i}+1)$ proves the proposition.

\item[(c)] In this case $M:=\displaystyle \prod
\limits_{i=0}^{k-1}\frac{q^{n-1-i}-1}{q^{i+1}-1}\cdot \displaystyle \prod
\limits_{i=n-k+1}^{n}(q^{i}+1)$ and $N:=\displaystyle \prod
\limits_{i=0}^{k-2}\frac{q^{n-2-i}-1}{q^{i+1}-1}\cdot \displaystyle \prod
\limits_{i=n-k+1}^{n-1}(q^{i}+1)$. 

\item[(d)] It follows from \cite[Chapter 23]{Thas} that $M:= \prod
\limits_{i=n-2k}^{n-1}(q^{i}-(-1)^{i})/\prod
\limits_{j=1}^{k}(q^{2j}-1)$
 and $N:= \prod
\limits_{i=n-2k}^{n-3}(q^{i}-(-1)^{i})/\prod
\limits_{j=1}^{k-1}(q^{2j}-1)$. 
\end{itemize}
\end{proof}

 \section{The dual code of $\mathcal{Q}(4,q)$, $q$ even}
\subsection{The maximum weight}

\begin{proposition}\label{pr}
If $c$ is a codeword of $\mathrm{C}(\q)^{\bot}$, $q=2^h$, then $wt(c)\leq q^3+q$, and if $wt(c)=q^3+q$, then $supp(c)$ is the complement of an ovoid.
\end{proposition}
\begin{proof}
Let $c$ be a codeword of $\mathrm{C}(\q)^{\bot}$. By condition $(**)$, the complement of $\S$ defines a blocking set $\B$ of $\mathcal{Q}(4,q)$. Hence, a
codeword of large weight corresponds to a blocking set of small
size. The smallest size for a blocking set of $\mathcal{Q}(4,q)$ is that of an ovoid, i.e. $q^{2}+1$. Moreover, by condition $(*)$, the complement of an ovoid defines
a codeword, and it has weight $(q+1)(q^{2}+1)-(q^2+1)=q^3+q$.
\end{proof}

\subsection{An empty interval in the weight distribution}
First, we extend Lemma \ref{anja} to weighted minihypers of $\q$.

\begin{lemma}\label{lemma2}
Let $(F,w)$ be a weighted $\{x(q+1),x;4,q\}$-minihyper, $x<\frac{q}{2}$,
contained in $\q$. Then $(F,w)$ is a sum of $x$ lines.
\end{lemma}
\begin{proof}
Let $k=\min\lbrace w(Q)\vert \vert Q \in F\rbrace$ and let $P$ be a
point with $w(P)=k$. Consider a tangent plane $\pi$ to $F$ in $P$, i.e. $\pi \cap F= \lbrace P\rbrace$. Let $S_{1},S_{2},\ldots,S_{q+1}$ be the $q+1$ solids through
$\pi$. At least one of them, say $S_{1}$, contains more than $x$
points of $F$ (counted according to their weight). By \cite[Lemma 2.1]{GS},
$S_{1} \cap F$ is a blocking set with respect to the planes of
$S_{1}$. Let $B$ be the minimal blocking set inside $S_{1} \cap F$. With the same arguments as in \cite[Lemma 3.2]{DBHS}, we get that $B$
contains a line, say $L$. Hence, we have a line $L$ of $\q$ completely
contained in $F$ through a point $P$ of minimum weight $k$. We construct
a minihyper $(F',w')$ in $\PG(4,q)$ in the following way: if $Q \in
L$, then $w'(Q)=w(Q)-1$, and $w'(Q)=w(Q)$ otherwise. By \cite[Lemma 2.2]{GS}, $(F',w')$ is a $\{(x-1)(q+1),x-1;4,q\}$-minihyper.
Repeating these arguments until all the points have weight zero, we
get that $(F,w)$ is a sum of $x$ lines.
\end{proof}

\begin{lemma}\label{lemma3}
Let $\mathcal{C}$ be a cover of $\q$ of size $q^{2}+1+r$ and let $E$
be the set of multiple points of $\mathcal{C}$. Then $(E,w)$, with $w(P)=
e(P)$, is a $\{r(q+1),r;4,q\}$-minihyper.
\end{lemma}
\begin{proof}
See the proof of \cite[Theorem 7, Part 1]{ESSS}.
\end{proof}

\begin{theorem}\label{spread}
Let $\mathcal{C}$ be a cover of $\q$, $q$ even, of size $q^{2}+1+r$,
where $0<r<\frac{q+4}{6}$. Then $\mathcal{C}$ contains a spread of $\q$.
\end{theorem}
\begin{proof}
 Lemma \ref{lemma1} shows that there exists a good line $L$ for the cover $\mathcal{C}$. Let $M_{1},M_{2},\ldots,M_{q+1}$ be the lines of
$\mathcal{C}$ intersecting $L$. Let $L^{\sigma}$ be the plane defined
by $L$ and by the nucleus of $\q$, then the planes
$L^{\sigma},\langle L,M_{1}\rangle,\langle L,M_{2}\rangle,\ldots,$
$\langle L,M_{q+1}\rangle$ define a
$(q+2)$-set $S$ in the quotient geometry $\PG(2,q)_L$ of $L$, such that
every line of $\PG(2,q)_L$ intersects $S$ in $0,1$ or $2$ points, except for at most $r$
lines which can contain, in total, at most $3r$ points of $S$ (see
Theorem 3 of \cite{ESSS}). Hence, at least $q+2-3r$ elements of $S$ are
internal nuclei. Since $q+2-3r>\frac{q}{2}$, every point of $S$ is
an internal nucleus (see \cite{BK}), i.e. $S$ has only $0$- and
$2$-secants. This implies that every hyperbolic quadric containing $L$
contains 0 or 2 lines of $\mathcal{C}$ intersecting $L$. By Lemmas \ref{lemma2} and \ref{lemma3}, the multiple points form a sum $\cal{L}$ of lines.

 Since $r>0$, there exist two intersecting lines $M_{1}$ and $M_{2}$ of $\mathcal{C}$. There are $q$ hyperbolic quadrics through $M_{1} \cup M_{2}$, where none of them contains a good line. Suppose that the cover $\mathcal{C}$ is minimal. Then the lines $M_{1}$ and $M_{2}$ are not contained in
$\mathcal{L}$. Let $\mathcal{Q}^{+}(3,q)$ be one of
the hyperbolic quadrics of $\q$ through $M_1\cup M_2$ and consider a regulus $\mathcal{R}$ of $\mathcal{Q}^{+}(3,q)$. A line of $\mathcal{R}$ cannot be a good line, hence, it is either a line of the cover $\mathcal{C}$ or it has at least one multiple point. Hence, we have at least $q+1$ multiple points in
$\mathcal{Q}^{+}(3,q)$. Since $q+1>2r$, at least one line of the sum $\mathcal{L}$ is contained in $\mathcal{Q}^{+}(3,q)$ and we know that it is not
$M_{1}$ nor $M_{2}$. So we find that $q\leq r$, leading us to a contradiction. This implies that the cover $\mathcal{C}$ is not minimal.

Since for every $0<r<\frac{q+4}{6}$, a cover $\mathcal{C}$ of size $q^2+1+r$ is not minimal, the minimal cover contained in $\mathcal{C}$ has size $q^2+1$, and hence, is a spread.
\end{proof}

\begin{corollary}\label{cor}
Let $B$ be a blocking set of $\q$, $q$ even, of size $q^{2}+1+r$,
where $0<r< \frac{q+4}{6}$. Then $B$ contains an ovoid of $\q$.
\end{corollary}
\begin{proof}
This follows from Theorems \ref{spread} and \ref{payne}(2).
\end{proof}
\begin{remark}
Theorem 8 of \cite{ESSS} shows that  a cover of $\q$, $q$ even, $q\geq 32$, of size $q^2+1+r$, $0<r\leq \sqrt{q}$, contains a spread of $\q$. Theorem \ref{spread} improves on this theorem for all values of $q$, $q$ even. Likewise, Corollary \ref{cor} improves on \cite[Theorem 9]{ESSS}.
\end{remark}

\begin{theorem} \label{gat} There are no codewords with weight in $]q^{3}+\frac{5q-4}{6},q^{3}+q[$ in $\mathrm{C}(\q)^{\bot}$, $q$ even.
\end{theorem}
\begin{proof}
Let $c$ be a codeword of $\mathrm{C}(\q)^{\bot}$ with weight in
$]q^{3}+\frac{5q-4}{6},q^{3}+q[$. This implies that $\B$ is a blocking set of $\q$ of size less than $q^{2}+1+\frac{q+4}{6}$. Corollary \ref{cor} shows that $\B$ contains an ovoid of $\q$, say $\mathcal{O}$. Let $c'$ be the codeword of weight $q^3+q$ of $\mathrm{C}(\q)^{\bot}$ defined by the complement of $\mathcal{O}$. Since $\mathrm{C}(\q)^{\bot}$ is a linear code, $c''=c+c'$ is a codeword of $\mathrm{C}(\q)^{\bot}$. Moreover, it has weight
at least 1 and less than $\frac{q+4}{6}$. This is a contradiction since the minimum weight of $\mathrm{C}(\q)^{\bot}$ is $2(q+1)$ (see \cite{Storme}).
\end{proof}

Let $\mathrm{C}(\mathcal{W}(q))^{\bot}$ be the dual binary code arising from the symplectic
polarity $\mathcal{W}(q)$, $q$ even. We can translate the results about $\mathrm{C}(\q)^{\bot}$ to $\mathrm{C}(\mathcal{W}(q))^{\bot}$ by using Theorem \ref{payne}(2).

\begin{theorem}
The largest weight of $\mathrm{C}(\mathcal{W}(q))^{\bot}$, $q$ even, is $q^{3}+q$,
which corresponds to codewords defined by the complement of an
ovoid, and there are no codewords of weight in
$]q^{3}+\frac{5q-4}{6},q^{3}+q[$.
\end{theorem}

We list some examples of large weight codewords.

\begin{example} Let $\mathcal{W}(q)$ be a symplectic polarity of $\PG(3,q)$. Now $\mathcal{W}(q)$, $q$ even, has ovoids $\mathcal{O}$, and they all satisfy the following properties: a line $L$ is
self-polar with respect to $\mathcal{W}(q)$ if and only if $L$ is a
tangent line of $\mathcal{O}$, and a line $M$, intersecting $\mathcal{O}$ in
two points, has $M^{\sigma}$ skew to $\mathcal{O}$ (see
\cite{Hirschfeld}). Let $c$ be the codeword defined by the complement
of $\mathcal{O}$ and let $c'$ be a codeword of minimum weight $2(q+1)$,
hence $c'$ is defined by two non-self-polar lines $M$ and
$M^{\sigma}$ (see \cite{Storme}). Then $c+c'$ is a codeword of $\mathrm{C}(\mathcal{W}(q))^{\bot}$ of
weight $q^{3}-q+2$.
\end{example}

\begin{example} Let $q$ be even. Let
$\AG(3,q)$ be the set of the affine points of $\PG(3,q)$, that is the
set of points of $\PG(3,q) \setminus \pi$, where $\pi$ is a plane:
every line of $\PG(3,q)$ contains 0 or $q$ points of $\AG(3,q)$, hence
$\AG(3,q)$ defines a codeword of $\mathrm{C}(\mathcal{W}(q))^{\bot}$, $q$ even, of weight $q^{3}$.
\end{example}

\begin{example} Let $q$ be even. Let
$c$ be a codeword of weight $q^{3}$ defined by the set
$\AG(3,q)=\PG(3,q) \setminus \pi$ and let $c'$ be a codeword of weight
$2(q+1)$ defined by two non-self-polar lines $L$ and $L^{\sigma}$, such that
$L \subseteq \pi$ and $L^{\sigma}$ intersects $\pi$ in one
point. The codeword $c+c'$ of $\mathrm{C}(\mathcal{W}(q))^{\bot}$, $q$ even, has weight $q^{3}+2$.
\end{example}

\begin{remark} By Theorem \ref{payne}, these examples correspond to codewords of weight $q^3+2$, $q^3$ or $q^3-q+2$ in $\mathrm{C}(\q)^{\bot}$, $q$ even. \end{remark}

A set  $\S$ of points such that every line of $\q$
contains an even number of points of $\mathcal{S}$ is often called {\em a set of even type}. This implies that results on the binary code $\mathrm{C}(\mathcal{Q}(4,q))^{\bot}$ are in fact results on sets of even type of $\q$.
\begin{corollary}The largest set of even type of $\q$, $q$ even, corresponds to the complement of an ovoid. There are no sets of even type of $\q$, $q$ even, with size in the interval $]q^{3}+\frac{5q-4}{6},q^{3}+q[$.

\end{corollary}

\section{The dual code of $\mathcal{Q}^+(5,q)$}
\subsection{Small weight codewords}

If $c$ is a codeword in the code $\C=\mathrm{C}_2(\mathcal{Q}^+(5,q))^{\bot}$, then $\S$ is a set of points of $\mathcal{Q}^+(5,q)$ such that every plane of $\mathcal{Q}^+(5,q)$ contains zero or at least 2 points of $\S$. Moreover, the sum of symbols $c_P$ of the points $P$ in each plane of $\mathcal{Q}^+(5,q)$ equals zero.
Using the Klein correspondence, this set $\S$ of points of $\mathcal{Q}^+(5,q)$ corresponds to a set $\mathcal{S}_L$ of lines in $\PG(3,q)$ such that:
\begin{enumerate}
\item[(1)] Every plane of $\PG(3,q)$ contains 0 or at least 2 lines of $\mathcal{S}_L$.
\item[(2)] Every point of $\PG(3,q)$ lies on 0 or on at least 2 lines of $\mathcal{S}_L$.
\item[(3)] The sum of the symbols of the lines of $\S_L$ going through a fixed point of $\PG(3,q)$ equals zero.
\item[(4)] The sum of the symbols of the lines of $\S_L$ lying in a fixed plane of $\PG(3,q)$ equals zero.
\end{enumerate}

\begin{example} \label{ex1}Let $\mathcal{S}_L$ be the set of $2q+2$ lines of a hyperbolic quadric $\mathcal{Q}^+(3,q)$ in $\PG(3,q)$, $q=p^h$, $p$ prime, $h\geq 1$, where all $q+1$ lines of one regulus get symbol $\alpha\in \mathbb{F}_p$, and the $q+1$ lines of the opposite regulus get symbol $-\alpha$. It is easy to check that this set $\mathcal{S}_L$ satisfies the conditions  (1)-(4). Under the Klein correspondence, the set $\mathcal{S}_L$ corresponds to a set $\S$ of points of $\mathcal{Q}^+(5,q)$, consisting of two conics, lying in two skew polar planes of $\mathcal{Q}^+(5,q)$. 
\end{example}
In Proposition \ref{H}, we will prove that the codewords of minimum weight in $\C$ correspond to this example.
\begin{example}\label{ex2} Let $\mathcal{S}_L$ be the set of $4q$ lines through two fixed points $P$ and $R$ of $\PG(3,q)$, $q=p^h$, $p$ prime, $h\geq 1$, lying in two fixed planes $\pi_1$ and $\pi_2$ through $PR$, different from the line $PR$, where all lines of $\S_L$ through $P$ in $\pi_1$ and all lines of $\S_L$ through $R$ in $\pi_2$ get symbol $\alpha$, and all other lines of $\S_L$ get symbol $-\alpha$. Under the Klein correspondence, the set $\mathcal{S}_L$ corresponds to a set $\S$ of $4q$ points in $\mathcal{Q}^+(5,q)$, lying on four lines through a fixed point $Q$, where these four lines define a quadrangle on the base $\mathcal{Q}^+(3,q)$ of the cone $T_Q(\mathcal{Q}^+(5,q))\cap \mathcal{Q}^+(5,q)$, where $T_Q(\mathcal{Q}^+(5,q))$ denotes the tangent hyperplane through $Q$.
\end{example}
In Theorem \ref{H2}, we will prove that the codewords of $\mathrm{C}(\mathcal{Q}^+(5,q))^{\bot}$ of weight at most $4q+4$ arise from Examples \ref{ex1} and \ref{ex2}, or of a linear combination of two codewords from Example \ref{ex1}.

\begin{remark}\label{opmm} A linear combination of two codewords from Example \ref{ex1} can have the following weight: $4q-4$, $4q-2$, $4q$, $4q+2$ or $4q+4$.
These numbers arise from the possible intersections of two hyperbolic quadrics $\mathcal{Q}^+(3,q)_1$ and $\mathcal{Q}^+(3,q)_2$ (see \cite{bruen}):
\begin{itemize}
\item $\mathcal{Q}^+(3,q)_1\cap \mathcal{Q}^+(3,q)_2$ equals 4 lines, two contained in a regulus $\R_1$ of $\mathcal{Q}^+(3,q)_1$ and a regulus $\R_2$ of $\mathcal{Q}^+(3,q)_2$, and two contained in the opposite reguli $\R_1^{opp}$ and $\R_2^{opp}$. Let $c_1$ be the codeword of $\mathrm{C}(\mathcal{Q}^+(5,q))^{\bot}$, where the lines of $\R_1$ get symbol $\alpha$ and the lines of $\R_1^{opp}$ get symbol $-\alpha$, and let $c_2$ be the codeword of $\mathrm{C}(\mathcal{Q}^+(5,q))^{\bot}$ where the lines of $\R_2$ get symbol $-\alpha$ and the lines of $\R_2^{opp}$ get symbol $\alpha$. Then $c_1+c_2$ is a codeword of $\mathrm{C}(\mathcal{Q}^+(5,q))^{\bot}$ with weight $4q-4$.

\item  $\mathcal{Q}^+(3,q)_1\cap \mathcal{Q}^+(3,q)_2$ equals 3 lines, two contained in a regulus $\R_1$ of $\mathcal{Q}^+(3,q)_1$ and a regulus $\R_2$ of $\mathcal{Q}^+(3,q)_2$, and one (with multiplicity 2) contained in the opposite reguli $\R_1^{opp}$ and $\R_2^{opp}$. Using the same ideas as in the preceding case, we can obtain a codeword of weight $4q-2$.

\item  $\mathcal{Q}^+(3,q)_1\cap \mathcal{Q}^+(3,q)_2$ equals 2 lines contained in a regulus $\R_1$ of $\mathcal{Q}^+(3,q)_1$ and a regulus $\R_2$ of $\mathcal{Q}^+(3,q)_2$. Using the same ideas, we can obtain a codeword of weight $4q$.
\item  $\mathcal{Q}^+(3,q)_1\cap \mathcal{Q}^+(3,q)_2$ equals 1 line (with multiplicity 2) contained in a regulus $\R_1$ of $\mathcal{Q}^+(3,q)_1$ and a regulus $\R_2$ of $\mathcal{Q}^+(3,q)_2$. Now we can obtain a codeword of weight $4q+2$.
\item $\mathcal{Q}^+(3,q)_1\cap \mathcal{Q}^+(3,q)_2$ contains no line. Now we obtain a codeword of weight $4q+4$.
\end{itemize}
\end{remark}
\begin{lemma}\label{veelofweinig} Let $\mathcal{S}_L$ be a set of lines in $\PG(3,q)$, satisfying conditions (1)-(2), with $\vert \mathcal{S}\vert\leq 4(q+1)$. A point of $\PG(3,q)$ lies on at most four or on at least $q-1$ lines of $\mathcal{S}_L$, when $q\geq 16$.
\end{lemma}
\begin{proof}
Let $R$ be a point of $\PG(3,q)$ and suppose that there are $x$ lines of $\mathcal{S_L}$ through $R$. Every plane through one of these $x$ lines of $\mathcal{S}_L$ through $R$ has to contain a second line of $\mathcal{S}_L$ because of condition (1). Since we assume that there are exactly $x$ lines of $\S_L$ through $R$, these extra lines do not pass through $R$, hence, we count any of those extra lines once. Only the planes spanned by two of the $x$ lines of $\S_L$ through $R$ do not necessarily need to contain an extra line. This gives in total at least 
$$x(q+1)-x(x-1)$$ extra lines, which can be at most $4(q+1)-x$. If $q\geq 16$, then this implies $x\leq 4$ or $x\geq q-1$.
\end{proof}

\begin{lemma} \label{l1} Let $\mathcal{S}_L$ be a set of lines in $\PG(3,q)$, $q\geq 19$, satisfying conditions (1)-(2), with $\vert \mathcal{S}_L\vert\leq 4(q+1)$. A point $R$ of $\PG(3,q)$ lying on at least 5 lines of $\S_L$ lies on at least $5q/3-38$ lines of $\S_L$.

 Moreover, $R$ lies in a plane $\pi_1$ with at least $2q/3-25$ lines of $\S_L$ through $R$ in $\pi_1$, and in a plane $\pi_2\neq \pi_1$ with at least $3q/7-18$ lines of $\S_L$ through $R$ in $\pi_2$. 
 
 Dually, if there is a plane with at least 5 lines of $\mathcal{S}_L$ lying in this plane, then this plane contains at least $5q/3-38$ lines of $\S_L$.
 
 \end{lemma}

\begin{proof} Let $R$ be a point of $\PG(3,q)$ on at least $5$ lines of $\mathcal{S}_L$. Lemma \ref{veelofweinig} implies that there are at least $q-1$ lines of $\mathcal{S}_L$ through $R$, say $L_1, L_2,\ldots, L_s$, $s\geq q-1$, hence there are $t$ points $R_i$, $i=1,\ldots,t$, $t\geq q(q-1)$, on $L_i$, $i=1,\ldots, s$, that have to lie on a second line of $\S$. There are at most $4q+4-q+1=3q+5$ lines in $\mathcal{S}\setminus \lbrace L_1,\ldots, L_s\rbrace$. Hence, there is a line with at least $$\frac{q(q-1)}{3q+5}\geq \frac{q-3}{3}$$ points of $\lbrace R_i\vert \vert i=1,\ldots,t\rbrace$. This implies that there is a plane $\pi$ through $R$ containing at least $(q-3)/3$ lines of $\S_L$ through $R$.

Suppose that $\pi$ contains $x\geq (q-3)/3$ lines $L_1,\ldots,L_x$ of $\S_L$ through $R$. Every plane through $L_i$, $i=1,\ldots,x$,  has to contain a second line of $\S_L$. Suppose there are $y$ lines through $R$, not in $\pi$. Then the number of extra lines of $\S_L$ needed is $xq-yx$, which has to be at most $4q+4-x-y$. This implies that 
$$\frac{(q+1)(x-4)}{x-1}\leq y.$$
Since $x$ is at least $(q-3)/3$, and $(q+1)(x-4)/(x-1)$ increases as $x$ increases, $y$ is at least $(q+1)(q-15)/(q-6)\geq q-13$ if $q\geq 19$.

Hence, we find at least $q-13$ lines of $\S_L$ through $R$, not in $\pi$.

This implies that there are at least $q-13+(q-3)/3$ lines of $\S_L$ through $R$. Repeating the previous argument, we get that there is a plane $\pi'$ through $R$ containing at least
$$\frac{q(4q/3-14)}{4q+4-4q/3+14}\geq \frac{q}{2}-10$$
lines of $\S_L$, and, again repeating the same calculations, that $R$ lies in total on at least $q/2-10+q-13$ lines of $\S_L$.

Repeating again, yields that $R$ lies on a plane $\pi''$ with at least $3q/5-20$ lines of $\S_L$ through $R$, and one last time, yields that $R$ lies on a plane $\pi_1$ with at least $2q/3-25$ lines  of $\S_L$ through $R$.

The same arguments show that there are at least $q-13$ lines of $\S_L$, not in the plane $\pi_1$, and that there is a plane $\pi_2$ with at least
 
$$\frac{(q-13)q}{4q+4-q+13-2q/3+25}\geq 3q/7-18$$ lines of $\S_L$ through $R$ in $\pi_2$.
Since the conditions (1)-(2) are self-dual, the dual of the first part of the statement holds.
\end{proof}

\begin{corollary} \label{g1}Let $\mathcal{S}_L$ be a set of lines in $\PG(3,q)$, $q\geq 51$, satisfying conditions (1)-(2), with $\vert \mathcal{S}_L\vert\leq 4(q+1)$. If there is a point $R$ lying on at least 5 lines of $\mathcal{S}_L$, then there are two planes through $R$ containing at least $5q/3-38$ lines of $\mathcal{S}_L$. Dually, if there is a plane $\pi$ with at least 5 lines of $\mathcal{S}_L$, then there are two points in $\pi$ on at least $5q/3-38$ lines of $\mathcal{S}_L$.
\end{corollary}
\begin{proof}
Since $q\geq 51$, the 2 planes $\pi_1$ and $\pi_2$ through $R$, found in  Lemma \ref{l1}, each contain more than 4 lines. Again using Lemma \ref{l1}, this implies that these two planes $\pi_1$ and $\pi_2$ each contain at least $5q/3-38$ lines of $\S_L$.
\end{proof}

\begin{lemma} \label{l0} Let $\mathcal{S}_L$ be a set of lines in $\PG(3,q)$, $q>124$, satisfying conditions (1)-(2), with $\vert \mathcal{S}_L\vert\leq 4(q+1)$. If there is a point $R$ lying on at least $5$ lines of $\S_L$, then $\S_L$ consists of the $4q$ lines through $R$ and a fixed point $S$ in two planes through $RS$, different from the line $RS$.
\end{lemma}
\begin{proof}

Since $R$ lies on at least 5 lines of $\S_L$, Corollary \ref{g1} implies that $R$ lies on two planes with at least $5q/3-38$ lines of $\S_L$. 

Suppose that there are 3 planes with at least this number of lines of $\S_L$. Then there are at least $3(5q/3-38)-3> 4q+4$ lines in $\S_L$ since $q>124$; a contradiction. This implies that there are exactly two planes, $\pi_1$ and $\pi_2$, containing more than 5 lines of $\S_L$, and dually, that there are exactly 2 points $R$ and $R'$ on at least 5 lines of $\S_L$. It follows from Corollary \ref{g1} that $R$ and $R'$ are contained in $\pi_1\cap \pi_2$.

Let $M_i$ be the lines of $\S_L$ that are contained in $\pi_1\cup \pi_2$, with $M_i\neq \pi_1\cap \pi_2$. Then the number of lines $M_i$ is at least $2(5q/3-39)$. Denote the intersection points of the lines $M_i$ with $\pi_1\cap \pi_2$ by $R$, $S_1$, $S_2$, $\ldots$.

Suppose that there is a point $S_j$ of $\pi_1\cap \pi_2$ such that all lines of $\S_L$ through it are contained in exactly one of the planes $\pi_1$ and $\pi_2$, say $\pi_1$. Let $M$ be a line of $\S_L$ through $S_j$, $M\neq \pi_1\cap\pi_2$. The $q$ planes through $M$, different from $\pi_1$, all contain a second line of $\S_L$, not in $\pi_1\cup \pi_2$. All points of those $q$ lines have to lie on another line of $\S_L$, not in $\pi_1\cup \pi_2$, which is not yet chosen. This implies that  there are $2q-1$ lines of $\S_L$, not in $\pi_1\cup \pi_2$; a contradiction, since $\vert \S_L\vert \leq 4q+4$, the number of lines of $\S_L$ in $\pi_1\cup \pi_2$ is at least $10q/3-77$, and $q>124$. Moreover, the same arguments prove that each point $S_j$ lies on a line of $\S_L$ in $\pi_2$, different from $\pi_1\cap \pi_2$.

Suppose that there are $x$ points $R$, $S_i$, $i=1,\ldots,x-1$, with $x\geq 3$. One of the planes, say $\pi_1$, has at most $2q+3$ lines of $\S_L$, since otherwise $\vert \S_L\vert>4q+4$. Hence, there are at least $xq-2q-3$ lines through $R$ and $S_i$, $i=1,\ldots,x-1$, in $\pi_1$ not in $\S_L$, so there is a point $R$ or $S_k$ lying on at least $(xq-2q-3)/x$ lines $N_1,\ldots, N_t$ in $\pi_1$ not of $\S_L$. As proven before, $R$ or $S_k$ lies on a line $M'$ of $\S_L$ in $\pi_2$ different from $\pi_1\cap \pi_2$. Since all planes through $M'$ and a line of $\lbrace N_1,\ldots,N_t\rbrace$ have to contain a second line of $\S_L$, not in $\pi_1$ and $\pi_2$, there are at least $(xq-2q-3)/x$ lines of $\S_L$ not in $\pi_1$ and $\pi_2$. Choosing one of these lines gives $q-1$ points that have to lie on lines of $\S_L$ that are not yet counted. This implies that there are at least $q-1+(xq-2q-3)/x$ lines of $\S_L$, not in $\pi_1\cup \pi_2$. If $q>124$ and $x\geq 3$, this is a contradiction.

Hence, there are only $2$ points $R$, $S_1$. Suppose that there is a line $L'$ of $\S_L$, not through $R$ or $S_1$.

 All planes through $L'$ have to contain a second line of $\S_L$, this implies that there are at least $q-1$ lines of $\S_L$, not in $\pi_1\cup \pi_2$. Let $L''$ be one of those lines. All points of $L''$, except for $L'\cap L''$, $L''\cap \pi_1$ and $L''\cap \pi_2$ have to lie on a second line of $\S_L$, which is not yet counted. This implies that there are at least $q-1+q-4=2q-5$ lines of $\S_L$, not in $\pi_1\cup \pi_2$. Since $q>124$, this is a contradiction.

This yields that all lines of $\S_L$ go through $R$ and $S_1$, and denote from now on $S_1=R'$. Suppose that there is a line $M_1$ of $\S_L$ not lying in $\pi_1\cup \pi_2$, and suppose w.l.o.g. that it contains $R$.  Every point of $M_1\setminus \lbrace R\rbrace$ lies on a second line of $\S_L$. This line can only go through $R'$. But then the plane $\langle M_1, R'\rangle$ contains at least $q+1$ lines of $\S_L$, a contradiction.

This implies that all lines of $\S_L$ are contained in the two planes $\pi_1$ and $\pi_2$. It is easy to see that $\S_L$ consists of all lines through $R$ and $R'$ in $\pi_1\cup \pi_2$, except for $\pi_1\cap\pi_2$. The line $\pi_1\cap \pi_2$ cannot be in $\S_L$, since in that case, any plane through $\pi_1\cap \pi_2$, different from $\pi_1$ and $\pi_2$, contains only the line $\pi_1\cap \pi_2$ of $\S$.
\end{proof}

\begin{lemma}\label{l2} Let $\mathcal{S}_L$ be a set of lines in $\PG(3,q)$, $q>88$, satisfying conditions (1)-(2), with $\vert \mathcal{S}_L\vert\leq 4(q+1)$. If there are no points lying on at least $5$ lines of $\S_L$, then $\S_L$ contains more than $q-6$ lines of each regulus of a hyperbolic quadric $\mathcal{Q}^+(3,q)$.
\end{lemma}
\begin{proof} Let $L$ be a line of $\S_L$, let $R_1,\ldots, R_{q+1}$ be the points of $L$. The point $R_i$ lies on a second line of $\S_L$, say $L_i$. If there are more than four lines of $\S_L$ in one plane, Corollary \ref{g1} shows that there is a point on more than 4 lines of $\S_L$. 

Suppose that two of the lines $L_i$, say $L_1$ and $L_2$, have a point in common. Then there can be at most one of the other lines $L_i$, say $L_3$, that is contained in $\langle L_1,L_2\rangle$, since otherwise, there would be more than four lines of $\S_L$ in this plane, a contradiction. From this, and the fact that three lines in a plane can meet in at most three points, we get that there are at most $3(q+1)$ points of the lines $L_i$, not in $L$, that are on another line $L_j$. This leaves at least $q^2-3(q+1)$ points $P_i$ on the lines $L_i$ that have to lie on a second line of $\S_L$. There are at most $4q+4-(q+2)=3q+2$ lines of $\S_L$ left. Hence, there is a line $L'$ containing at least $\frac{q^2-3(q+1)}{3q+2}\geq (q-4)/3$ of the points $P_i$. Note that the line $L'$ is skew to $L$, otherwise there would be a plane with at least $(q-4)/3+1$ lines of $\S_L$ in it, a contradiction.

Let $L_1,\ldots, L_s$ be the $s\geq (q-4)/3$ lines of $\S_L$ intersecting both $L$ and $L'$. On $L_i, i=1,\ldots,s$, there are at least $(q-1)s$ points $Q_k$ that have to lie on a second line of $\S_L$; at most $2(q+1-s)$ of them lie on one of the lines $L_{s+1},\ldots,L_{q+1}$. This gives at least $(q-1)s-2(q+1-s)=(s-2)(q+1)\geq (q+1)((q-10)/3)$ points that have to lie on one of the $4q+4-(q+3)=3q+1$ lines of $\S_L\setminus(\lbrace L,L'\rbrace \cup \lbrace L_i\vert\vert i=1,\ldots,q+1\rbrace)$.
Hence, there is a line $L''$ containing at least 
$$\frac{(q+1)((q-10)/3)}{3q+1}\geq (q-10)/9$$ points of $\lbrace Q_k\vert\vert k=1,\ldots,t\geq (q-1)s\rbrace$.

So we find three skew lines $L,L',L''$, defining a hyperbolic quadric $\mathcal{Q}=\mathcal{Q}^+(3,q)$ with in one regulus at least $(q-10)/9$ lines of $\S_L$.
Suppose there are $x$ lines of $\S_L$ in this regulus of $\mathcal{Q}$ and $t$ lines of $\S_L$ in the opposite regulus of $\mathcal{Q}$. Let $x\geq t$. This implies that $(q+1-x)t+(q+1-t)x$ points of $\Q$ have to lie on a second line of $\S_L$. This number is at most $2(4q+4-x-t)$, since a line of $\S_L$, not in $\Q$, can intersect $\Q$ in at most 2 points. From
\begin{eqnarray}
(q+1-x)t+(q+1-t)x\leq 2(4q+4-x-t),\label{vgl}
\end{eqnarray}
and $x,t\leq q+1$, we get that
$$x-t\leq 8,$$
and since $x\geq (q-10)/9$, that $t\geq (q-10)/9-8$. Set $x=t+i$, $i=0,\ldots,8$.

From inequality (\ref{vgl}), it follows that
$$ (q+1-t-i)t+(q+1-t)(t+i)\leq 8q+8-2t-2i-2t,$$
hence that
\begin{eqnarray}
-2t^2+2t(q+3-i)+(i-8)(q+1)+2i\leq 0.\label{vgl2}
\end{eqnarray}
Let $t=(q-10)/9-8$. Filling in $i=0,\ldots, 8$ yields a contradiction for $q>88$.
Let $t=q-6$. Filling in $i=0,\ldots,8$ yields a contradiction for $q>13$.

Hence, $\Q$ is a hyperbolic quadric with more than $q-6$ lines of $\S_L$ in each regulus.
\end{proof}

\begin{proposition} \label{H} Let $c$ be a codeword of weight at most $2q+2$ of $\C$, $q>88$, then $\S_L$ is a hyperbolic quadric $\mathcal{Q}^+(3,q)$ in $\PG(3,q)$, with all lines in one regulus having symbol $\alpha$, and all lines in the opposite regulus symbol $-\alpha$.
\end{proposition}
\begin{proof} As seen in the introduction to this section, a codeword $c$ of $\C$ corresponds to a set $\S_L$ of lines in $\PG(3,q)$, satisfying conditions $(1)-(4)$. Looking at all planes of $\PG(3,q)$  through a line of $\S_L$ shows that $d(\C)\geq q+2$.

 It follows from Lemma \ref{l0}  that $\S_L$ is a set of lines such that there is no point of $\PG(3,q)$ lying on more than 4 lines of $\S_L$. Lemma \ref{l2} shows that at least $2q-12$ lines of $\S$ are contained in a hyperbolic quadric $\Q$ of $\PG(3,q)$; at least $q-6$ lines $L_i$, $i=1,\ldots,s$, $s\geq q-6$, of $\S_L$ in a regulus $\mathcal{R}$ and  at least $q-6$ lines $M_j$, $j=1,\ldots,t$, $t\geq q-6$, in the opposite regulus $\mathcal{R}'$ of $\Q$.

 Since $2q+2-2(q-6)=14$, there are at most 14 lines of $\S_L$ not contained in $\Q$, which gives at most $28$ points $Q_i$ on $\Q$ that lie on lines of $\S_L$, not contained in $\Q$. Suppose that some line of $\Q$ does not belong to $\S_L$. There is a line, say $L_1$, of $\lbrace L_i\vert\vert i=1,\ldots,s\rbrace$ containing none of the points $Q_i$, since $q-6>28$. Suppose that $L_1$ has symbol $\alpha$ in the corresponding codeword $c$ of $\mathrm{C}(\mathcal{Q}^+(5,q))^{\bot}$. Then the lines of $\lbrace M_i\vert\vert i=1,\ldots,t\rbrace$, which all intersect $L_1$, all have symbol $-\alpha$. Since $q-6>28$, there is a line, say $M_x$, of $\lbrace M_i\vert\vert i=1,\ldots,t\rbrace$ containing none of the points $Q_i$. Then the lines of $\lbrace L_i\vert\vert i=1,\ldots,s\rbrace$, which all intersect $M_x$, all have symbol $\alpha$.

Give all lines of the regulus $\mathcal{R}$ of $\mathcal{Q}$ containing $L_1$ the symbol $\alpha$, and give all lines of the opposite regulus $\mathcal{R}'$ the symbol $-\alpha$. As seen in Example \ref{ex1}, this set of lines corresponds to a codeword $c'$ of $\C$. 

If $c\neq c'$, then $c-c'$ is a non-zero codeword of the linear code $\C$ with weight at most $wt(c)+wt(c')-2wt(c\cap c')\leq 2q+2+2q+2-2(2q-12)=28$.

Since $28<q+2$, this is a contradiction. Hence, $c=c'$ and the theorem is proven.
\end{proof}

\begin{theorem}\label{H2} Let $c$ be a codeword of weight at most $4q+4$ of $\C$, $q>124$, then $\S=supp(c)$ corresponds via the Klein correspondence to one of the following configurations of lines in $\PG(3,q)$:
\begin{enumerate}
\item a hyperbolic quadric $\mathcal{Q}^+(3,q)$ in $\PG(3,q)$, with all lines in one regulus symbol $\alpha$, and all lines in the opposite regulus symbol $-\alpha$,
\item a linear combination of two codewords of type 1,
\item $4q$ lines through two fixed points $R$ and $S$ in two planes $\pi_1$  and $\pi_2$ through $RS$, the line $RS$ not included, where the lines through $R$ in $\pi_1$ and the lines through $S$ in $\pi_2$ have symbol $\beta$, and the other lines have symbol $-\beta$.
\end{enumerate}
\end{theorem}
\begin{proof} As seen in the introduction to this section, a codeword $c$ of $\C$ corresponds to a set $\S_L$ of lines in $\PG(3,q)$, satisfying conditions $(1)-(4)$. If there is a point in $\PG(3,q)$ lying on at least 5 lines of $\S_L$, Lemma \ref{l0} shows that $\S_L$ is the set of the $4q$ lines through two fixed points $R$ and $S$ in two planes $\pi_1$  and $\pi_2$ through $RS$, the line $RS$ not included. If a line of $\S_L$ through $R$ in $\pi_1$ has symbol $\beta$, it is easy to see that all lines  of $\S_L$ through $R$ in $\pi_2$, and all lines of $\S_L$ through $S$ in $\pi_1$, have symbol $-\beta$. The same argument implies that the lines of $\S_L$ through $S$ in $\pi_2$ have symbol $\beta$. This proves part (3) of the statement.

If all points of $\PG(3,q)$ lie on at most $4$ lines of $\S_L$, Lemma \ref{l2} shows that there is a hyperbolic quadric $\Q$ with at least $q-6$ lines $L_i$, $i=1,\ldots,s$, $s\geq q-6$, of $\S_L$ in a regulus $\mathcal{R}$ and  at least $q-6$ lines $M_j$, $j=1,\ldots,t$, $t\geq q-6$, of $\S_L$ in the opposite regulus $\mathcal{R}'$ of $\Q$. There are at most $4q+4-(2q-12)=2q+16$ lines of $\S_L$ not on $\Q$. This gives in total at most $4q+32$ points $Q_i$ on $\Q$ lying on lines of $\S_L$, not contained in $\Q$. 

Suppose that each of the $q-6$ lines of $\mathcal{R}\cap \S_L$ contains at least $7$ points $Q_i$, then the number of points $Q_i$ would be at least $7(q-6)>4q+32$, a contradiction if $q>24$. Hence, there is a line, say $L_1$, with at most 6 such points $Q_i$. Suppose that $L_1$ has symbol $\alpha$, then there are at least $q-6-6$ lines $M_i$ with symbol $-\alpha$. There is one of these $t'\geq q-12$ lines $M_1,\ldots,M_{t'}$ containing at most $6$ points $Q_i$. Suppose that all lines $M_1,\ldots,M_{t'}$ contain at least 7 points $Q_i$, then the number of points $Q_i$ is at least $7(q-12)$, which is larger than $4q+32$ if $q>38$, a contradiction.
This implies that at least $q-12$ lines $L_i$ have symbol $\alpha$.

Give all lines of $\mathcal{R}$ the symbol $\alpha$ and give all lines of $\mathcal{R}'$ the symbol $-\alpha$. As seen in Example \ref{ex1}, this set of lines corresponds to a codeword $c'$ of $\C$. 

If $c\neq c'$, then $c-c'$ is a non-zero codeword of the linear code $\C$ with weight at most $wt(c)+2q+2-2(2q-24)=wt(c)-2q+50<wt(c)$, if $q\geq27$. Hence, the codeword $c-c'$ is a codeword of weight at most $2q+54< 4q-4$ since $q>29$, hence there is no point lying on at least 5 lines of the line set $\S_L'$ corresponding to $c-c'$. Two hyperbolic quadrics intersect in at most 4 lines (see Remark \ref{opmm}), so a linear combination of two codewords arising from a hyperbolic quadric has weight at least $4q-4$. Hence by induction on the weight of the codewords, the codeword $c-c'$ is a hyperbolic quadric with weight $2q+2$, so $c$ is a linear combination of two codewords of type (1).
\end{proof}

\begin{remark} Note that there are two different kinds of codewords of weight $4q$, i.e., the codewords arising from a linear combination of two codewords of weight $2q+2$ and the codewords of the third type in Theorem \ref{H2}, arising from 4 lines in two fixed planes.
\end{remark}

\subsection{Large weight codewords for $q$ even}
 
As seen in the previous section, a codeword $c$ in the binary code $\C$, $q=2^h$, defines a set $S$ of points of $\mathcal{Q}^+(5,q)$ such that every plane of $\mathcal{Q}^+(5,q)$ contains an even number of points of $S$. Hence, $\B=\mathcal{Q}^+(5,q)\setminus S$ is a set of points such that every plane of $\mathcal{Q}^+(5,q)$ contains at least one point of $\B$. Such a set is a {\em blocking set} of $\mathcal{Q}^+(5,q)$. If every plane contains exactly one point of $\B$, this set is called an {\em ovoid} of $\mathcal{Q}^+(5,q)$. These ovoids exist (see \cite{Thas}) and have size $q^2+1$. This implies that the codewords of maximal weight in $\C$, $q$ even,  correspond to the complement of an ovoid, hence have size $\vert \mathcal{Q}^+(5,q)\vert-(q^2+1)=(1+q^2)(q^2+q)$.

If $wt(c)=(1+q^2)(q^2+q)-r$, then $\B$ is a blocking set of size $q^2+1+r$ meeting every plane of $\mathcal{Q}^+(5,q)$ in an odd number of points since $q$ is even. Using the Klein correspondence, this set of points of $\mathcal{Q}^+(5,q)$ corresponds to a set $\mathcal{S}_L$ of lines in $\PG(3,q)$ such that:
\begin{enumerate}
\item[(1)] every plane of $\PG(3,q)$ contains an odd number of lines of $\mathcal{S}_L$.
\item[(2)] every point of $\PG(3,q)$ lies on an odd number of lines of $\mathcal{S}_L$.
\end{enumerate}

A {\em cover} $\mathcal{C}$ of $\PG(3,q)$ is a set $\mathcal{L}$ of lines such that every point of $\PG(3,q)$ lies on at least one line of $\mathcal{L}$. For a study of covers of $\PG(3,q)$, we refer to \cite{blokhuis}.
\begin{lemma} A codeword of $\C$, $q$ even, has even weight.
\end{lemma}

\begin{proof}
Let $c$ be a codeword of $\C$, with $wt(c)=(1+q^2)(q^2+q)-r$, then $\vert \mathcal{S}_L\vert=q^2+1+r$, and $\mathcal{S}_L$ defines a cover of size $q^2+1+r$.
A double counting of the number of pairs ($P\in \PG(3,q)$, line $L$ of $\mathcal{S}_L$ through $P$) yields that $r(q+1)$ is the sum of the excesses of the multiple points. Since every point of $\PG(3,q)$ lies on an odd number of lines of $\S_L$, every point has even excess, so in total, the sum of all the excesses is even. Since $q$ is even, this implies that $r$ is even, hence that $wt(c)$ is even.
\end{proof}

\begin{theorem}\label{constr} There are codewords in $\mathrm{C}(\mathcal{Q}^+(5,q))^{\bot}$, $q$ even, of weight $(1+q^2)(q^2+q)-2i$, where $i=0,1,\ldots,q/2$.
\end{theorem}
\begin{proof}
We give an explicit construction of these codewords. Let $T$ be a regular spread of $\PG(3,q)$, let $L$ be a line of $T$ and let $\mathcal{R}_1,\ldots,\mathcal{R}_q$ be $q$ reguli of $T$ through $L$ which pairwise only share $L$. Replace $2i$ of the reguli $\mathcal{R}_1,\ldots,\mathcal{R}_q$ by their opposite reguli. Put the line $L$ back. Let $\S_L$ be the set  of $(q-2i)q+1+2i(q+1)=q^2+1+2i$ lines obtained in this way.

Let $\pi$ be a plane in $\PG(3,q)$ through $L$. The plane $\pi$ cannot contain another element of $T$. Let $\mathcal{R}$ be one of the reguli through $L$ which is replaced by its opposite regulus. Then there is exactly one transversal line to $\mathcal{R}$ contained in $\pi$. Hence, a plane through $L$ contains exactly $2i+1$ lines of $\S_L$. 

Let $\pi'$ be a plane in $\PG(3,q)$, not through $L$. It contains exactly one line $L'$ of $T$, and there is exactly one regulus of $\mathcal{R}_1,\ldots,\mathcal{R}_q$ containing the line $L'$. If this regulus is replaced by its opposite regulus, there is a transversal line $t$ through $L$ and $L'$ contained in $\pi'$. Moreover, if another regulus $\mathcal{R}'$ has a transversal line $t'$ contained in $\pi'$, $t$ and $t'$ intersect, and $\mathcal{R}=\mathcal{R}'$, a contradiction. We conclude that Condition (1) holds since every plane through $L$ contains exactly $2i+1$ lines of $\S_L$,  and a plane, not through $L$, contains exactly one line of $\S_L$.

Condition (2) holds since a point not on $L$ lies on exactly one line of $\S_L$, while a point of $L$ lies on $2i+1$ lines of $\S_L$. 

This implies that via the Klein correspondence the complement of $\S_L$ is a codeword of $\C$ of weight $(1+q^2)(q^2+q)-2i$.
\end{proof}
\begin{remark}
It is interesting to notice the difference between the possible large weight codewords in $\mathrm{C}(\q)^{\bot}$, $q$ even, and $\mathrm{C}(\mathcal{Q}^+(5,q))^{\bot}$, $q$ even. In $\mathrm{C}(\q)^{\bot}$, $q$ even, Theorem \ref{gat} shows there is an empty interval in the weight enumerator, whereas Theorem \ref{constr} constructs codewords in $\mathrm{C}(\mathcal{Q}^+(5,q))^{\bot}$, $q$ even, for every even value in $[(1+q^2)(q^2+q)-q,(1+q^2)(q^2+q)]$.
\end{remark}

\section{The dual code of $\mathcal{H}(5,q^2)$}

If $c$ is a codeword in the $p$-ary code $\Ha=\mathrm{C}_2(\mathcal{H}(5,q^2))^{\bot}$, $q=p^h$, $p$ prime, $h\geq 1$, $\S$ is a set of points of $\mathcal{H}(5,q^2)$ such that every plane of $\mathcal{H}(5,q^2)$ contains zero or at least 2 points of $\S$. Moreover, the sum of the symbols $c_P$ of the points $P$ in a plane of $\mathcal{H}(5,q^2)$ equals zero. 

\begin{example} \label{vb}Let $\Gamma$ be a Hermitian curve $\mathcal{H}(2,q^2)\subseteq \mathcal{H}(5,q^2)$, lying in the plane $\pi\not\subseteq \mathcal{H}(5,q^2)$, let $\Gamma'$ be the Hermitian curve $\mathcal{H}(5,q^2)\cap \pi^{\sigma}$, with $\sigma$ the Hermitian polarity defined by $\mathcal{H}(5,q^2)$. Let $\mu$ be a plane of $\mathcal{H}(5,q^2)$ through a point $Q\in \Gamma$. Since $\pi^{\sigma}\subset Q^{\sigma}$, $\mu=\mu^{\sigma}\subset Q^{\sigma}$, and  $\mu\subseteq \mathcal{H}(5,q^2)$, the planes $\pi^{\sigma}$ and $\mu$ intersect in at least one point of $\mathcal{H}(5,q^2)$.

Hence, $\S=\Gamma\cup\Gamma'$ is a set of $2(q^3+1)$ points such that every plane of $\mathcal{H}(5,q^2)$ contains zero or at least two points of $\S$. Giving all points of $\Gamma$ symbol $\alpha$ and all points of $\Gamma'$ symbol $-\alpha$, yields a codeword of $\Ha$ of weight $2(q^3+1)$.
\end{example}

\begin{example} \label{vb'} Let $\pi$ be a plane of $\PG(5,q^2)$ intersecting $\mathcal{H}(5,q^2)$ in a cone $\Gamma$ with vertex $P$ and base a Baer subline. Let $\Gamma'$ be the intersection of $\pi^{\sigma}$ with $\mathcal{H}(5,q^2)$. It is easy to show that $\S=( \Gamma \cup \Gamma') \setminus \lbrace P\rbrace$ is a set such that every plane of $\mathcal{H}(5,q^2)$ contains zero or at least 2 points of $\S$. Giving all points of $\Gamma\setminus\lbrace P \rbrace$ symbol $\alpha$, all points of $\Gamma'\setminus \lbrace P \rbrace$ symbol $-\alpha$, and the point $P$ symbol zero, yields a codeword of weight $2(q^3+q^2)$ in $\Ha$.
\end{example}

We will now characterise the two smallest weight codewords of $\mathrm{C}(\mathcal{H}(5,q^2))^{\bot}$ to be the codewords of Example \ref{vb} and Example \ref{vb'} (see Theorem \ref{herm}).

\begin{lemma} \label{le1}Let $c$ be a codeword of $\Ha$ with weight at most $2(q^3+q^2)$, let $\pi$ be a plane intersecting $\mathcal{H}(5,q^2)$ in a Hermitian curve $\Gamma\cong \mathcal{H}(2,q^2)$, containing $x$ points of $\S$, and $\Gamma'\cong \mathcal{H}(2,q^2)=\pi^{\sigma}\cap \mathcal{H}(5,q^2)$, containing $t$ points of $\S$.
Then $\max(x,t)\leq (13q+19)/2$ or $\min(x,t)\geq (2q^3-5q-5)/2$ if $q>327$. \end{lemma} 
\begin{proof}
Let $c$ be a codeword of $\Ha$ with weight at most $2(q^3+q^2)$, let $\pi$ be a plane intersecting $\mathcal{H}(5,q^2)$ in a Hermitian curve $\Gamma\cong \mathcal{H}(2,q^2)$, containing $x$ points of $\S$, and $\Gamma'\cong \mathcal{H}(2,q^2)=\pi^{\sigma}\cap \mathcal{H}(5,q^2)$, containing $t$ points of $\S$. Consider a line $L$ through a point of $\Gamma \cap \S$ and a point of $\Gamma'\setminus \S$. Then there are $x(q^3+1-t)$ such lines, where each such line lies in $q+1$ planes of $\mathcal{H}(5,q^2)$, which yields $x(q^3+1-t)(q+1)$ planes of $\mathcal{H}(5,q^2)$ passing through these lines. All these planes need an extra point of $\S\setminus (\pi_1\cup \pi_2)$.

A point $R$ of $\mathcal{H}(5,q^2)\setminus (\pi_1\cup \pi_2)$ lies on a unique line $M=\langle \pi,R\rangle \cap \pi^{\sigma}$, intersecting both $\pi$ and $\pi^{\sigma}$. Let $M\cap \pi=\{\pi_R\}$ and $M\cap \pi^{\sigma}=\{\pi_R'\}$. The intersection of $R^{\sigma}$ with $\pi$ is a line $L$, and the intersection of $R^{\sigma}$ with $\pi^{\sigma}$ is a line $L'$. On $L$ and $L'$, there are at most $q+1$ points of $\mathcal{H}(5,q^2)$ collinear with $R$. If $M\not\subseteq \mathcal{H}(5,q^2)$, this implies that $R$ lies on at most $(q+1)^2$ planes of $\mathcal{H}(5,q^2)$ through a point of $\Gamma$ and $\Gamma'$. If $M\subseteq \mathcal{H}(5,q^2)$, there are $q+1$ planes of $\mathcal{H}(5,q^2)$ through $M$ that do not contain a point of $\pi\setminus \{\pi_R\}$, and $\pi^{\sigma}\setminus \{\pi_R'\}$, so there are at most $q^2$ planes of $\mathcal{H}(5,q^2)$ through $R$, and a point of $\Gamma$ and $\Gamma'$, which gives in total at most $q^2+q+1$ planes of $\mathcal{H}(5,q^2)$ through $R$ and a point of $\Gamma$ and $\Gamma'$. We conclude that every point $R$ of $\mathcal{H}(5,q^2)\setminus (\pi_1\cup \pi_2)$ blocks at most $(q+1)^2$ planes through a point of $\Gamma$ and a point of $\Gamma'$.

So we need at least $x(q^3+1-t)(q+1)/(q+1)^2$ points in $\S\setminus (\pi_1\cup \pi_2)$. Doing the symmetrical calculation starting from the lines $L$ through a point of $\Gamma\setminus \S$ and a point of $\Gamma'\cap \S$, and taking the average of both calculations and adding the $x+t$ points of $(\Gamma\cup \Gamma')\cap \S$, yields that there are in total at least 
\begin{eqnarray}
x(q^3+1-t)/(2q+2)+t(q^3+1-x)/(2q+2)+x+t\leq 2(q^3+q^2)\label{0}
\end{eqnarray}
 points in $supp(c)$. 

Since $t$ is at most $q^3+1$, and omitting the term $x+t$, this gives that
\begin{eqnarray}
(x+t)(q^3+1)-2x(q^3+1)\leq 4(q+1)(q^3+q^2).\label{1}
\end{eqnarray}
Suppose that $\min(x,t)=x$, otherwise switch $t$ and $x$. If we substitute $t=x+i$, then inequality (\ref{1}) yields that $i<4q+9$.

Rewriting inequality (\ref{0}) shows that 
$$2x^2-2x(q^3+2q+3-i)-i(q^3+2q+3)+4q^2(q+1)^2\geq 0.$$
 This implies, together with $0\leq i\leq 4q+8$, that $x\leq (5q+3)/2$, or $x\geq (2q^3-5q-5)/2$.
 \end{proof}
 
 \begin{lemma}\label{le2}Let $c$ be a codeword of $\Ha$ with weight at most $2(q^3+q^2)$, let $\pi$ be a plane intersecting $\mathcal{H}(5,q^2)$ in a cone $\Gamma$ with vertex $P$ and base a Baer subline, containing $x$ points of $\S$, and $\Gamma'=\pi^{\sigma}\cap \mathcal{H}(5,q^2)$ a cone with vertex $P$ and base a Baer subline in $\pi^{\sigma}$, containing $t$ points of $\S$.
Then $\max(x,t)\leq (17q^2-270q+10)/2$ or $\min(x,t) \geq q^3-6q^2$ if $q>327$. \end{lemma} 
 \begin{proof}
 
We consider the lines $L$ of $\mathcal{H}(5,q^2)$  through a point of $(\Gamma\cap \S)\setminus \lbrace P \rbrace$ and a point of $\Gamma'\setminus (\lbrace P \rbrace\cup \S)$. There are at least $(x-1)(q^3+q^2-1-t)$ such lines, and all such lines $L$ lie in $q+1$ planes of $\mathcal{H}(5,q^2)$. One of these planes is $\langle L,P\rangle$, but we do not consider this plane.

We have at least $(x-1)(q^3+q^2-1-t)q$ such planes, only intersecting $\pi$ and $\pi^{\sigma}$ in one point, different from $P$. Each of those planes needs at least one extra point of $\S$, which lies in $\S\setminus (\pi\cup \pi^{\sigma})$.
Let $R$ be a point of $\S\setminus (\pi\cup \pi^{\sigma})$. If $R\notin P^{\sigma}$, $R^{\sigma}$ intersects $\pi$ and $\pi^{\sigma}$ in lines, containing $q+1$ points of $\mathcal{H}(5,q^2)$. Hence, in this case, $R$ lies on at most $(q+1)^2$ planes of $\mathcal{H}(5,q^2)$ intersecting $\Gamma$ and $\Gamma'$. If $R\in P^{\sigma}$, $R^{\sigma}$ intersects $\pi$ (resp. $\pi^{\sigma}$) in a line $L$ (resp. $L'$) through $P$. If $L$ or $L'$ is not a line of $\mathcal{H}(5,q^2)$, $R$ cannot lie in a plane of $\mathcal{H}(5,q^2)$ through a point of $\Gamma$ and $\Gamma'$ different from $P$. Hence, suppose that both $L$ and $L'$ are lines of $\mathcal{H}(5,q^2)$. The planes $\langle R,L\rangle$ and $\langle R,L'\rangle$ are contained in $\mathcal{H}(5,q^2)$, hence, they coincide. So this gives $q^2$ lines through $R$, intersecting $L$ and $L'$, different from the line $RP$, where each of these lines lies in $q$ planes of $\mathcal{H}(5,q^2)$, different from $\langle L,L'\rangle$. This gives in total at most $q^3$ planes through $R$ intersecting $\Gamma$ and $\Gamma'$, in a point different from $P$.

 Hence, we need at least
$(x-1)(q^3+q^2-1-t)q/q^3$ points in $\S\setminus (\pi\cup \pi^{\sigma})$. Doing again the symmetrical calculation and taking the average yields that there are at least 
\begin{eqnarray}
(x-1)(q^3+q^2-1-t)/(2q^2)+(t-1)(q^3+q^2-1-x)/(2q^2)+x+t\leq 2(q^3+q^2)\label{0'}
\end{eqnarray}
 points in $\S$. 

Since $t$ is at most $q^3+q^2+1$, and omitting the term $x+t$, this gives that
\begin{eqnarray}
(x+t-2)(q^3+q^2-1)-2x(q^3+q^2+1)\leq 4q^2(q^3+q^2).\label{1'}
\end{eqnarray}
Suppose that $\min(x,t)=x$, otherwise switch $t$ and $x$. If we substitute $t=x+i$, then inequality (\ref{1'}) yields that $i<4q^2+5$.

Rewriting inequality (\ref{0'}) shows that 
$$2x^2-2x(q^3+3q^2-i)-(1+i-iq^3-3iq^2)+4q^2(q^3+q^2)\geq 0.$$
 This implies, together with $0\leq i\leq 4q^2+5$, that $x\leq 9(q^2-30q)/2$ or $x\geq q^3-6q^2$ if $q>327$.
 \end{proof}

\begin{lemma} \label{vlakvanH} Let $c$ be a codeword of $\Ha$ of weight at most $2(q^3+q^2)$.
A plane $\pi$ of $\mathcal{H}(5,q^2)$ contains at most $2q^2+2q+2$ points of $\S$.
\end{lemma}
\begin{proof}
Let $\pi$ contain $x$ points of $\S$. Let $P$ be a point of $\pi\cap \S$. If we project from $P$ onto its quotient geometry $\mathcal{H}(3,q^2)$, then the points of $supp(c)\cap P^{\sigma}$ are projected onto a blocking set of $\mathcal{H}(3,q^2)$ with respect to lines, and the points of $\pi$ are projected onto a line $L$ of $\mathcal{H}(3,q^2)$. There are $(1+q)(q^3+1)-1-(q^2+1)q=q^4$ lines of $\mathcal{H}(3,q^2)$ skew to $L$. A point $R\notin L$ lies on $q$ of those lines, so we need at least $q^3$ extra points to block the lines of $\mathcal{H}(3,q^2)$ skew to $L$.

Hence, $ P^{\sigma}$ contains at least $q^3$ points of $\S$ outside of $\pi$.

We count the number of pairs $(P,R)$, $P\in \pi\cap \S$, $R\in \S\setminus \pi$, with $PR$ a line of $\mathcal{H}(5,q^2)$.
There are at least $xq^3$ such pairs, moreover, a point $R\in \S\setminus \pi$ is collinear with $q^2+1$ points of $\pi$, hence, there are at most $2(q^3+q^2)(q^2+1)$ such pairs. This implies that
$$xq^3\leq 2(q^3+q^2)(q^2+1),$$
from which it follows that $x\leq 2q^2+2q+2$.
\end{proof}


\begin{lemma} \label{?} Let $c$ be a codeword of $\Ha$, $q>893$, of weight at most $2(q^3+q^2)$, then there is a plane $\pi\not\subseteq \mathcal{H}(5,q^2)$ containing more than $(17q^2-270q+10)/2$ points of $\S$.
\end{lemma}
\begin{proof}
Let $c$ be a codeword of $\Ha$ of weight at most $2(q^3+q^2)$ and let $P_i$, $i=1,\ldots,3$, be points of $\S$. As seen in Proposition \ref{minimum}(d), the set $P_i^{\sigma}\cap \S$ contains at least $q^3+2$ points. Using this lower  bound, together with the fact that $\vert (P_1^{\sigma}\cup P_2^{\sigma} \cup P_3^{\sigma})\cap \S\vert$ is at most $2(q^3+q^2)$, and
$$\vert (P_1^{\sigma}\cup P_2^{\sigma} \cup P_3^{\sigma})\cap \S\vert=$$
$$ \vert P_1^{\sigma}\cap \S\vert +  \vert P_2^{\sigma}\cap \S\vert+ \vert P_3^{\sigma}\cap \S\vert- \vert P_1^{\sigma}\cap P_2^{\sigma} \cap \S\vert$$
$$-\vert P_1^{\sigma}\cap P_3^{\sigma} \cap \S\vert-\vert P_2^{\sigma}\cap P_3^{\sigma} \cap \S\vert+\vert P_1^{\sigma}\cap P_2^{\sigma}\cap P_3^{\sigma} \cap \S\vert$$
yields that $\max(\vert P_1^{\sigma}\cap P_2^{\sigma} \cap \S\vert, \vert P_1^{\sigma}\cap P_3^{\sigma} \cap \S\vert,\vert P_2^{\sigma}\cap P_3^{\sigma} \cap \S\vert)$ is at least $(q^3-2q^2+6)/3$.

Hence, for any three points $Q_i$, $i=1,\ldots,3$, in $\S$, $\vert Q_i^{\sigma}\cap Q_j^{\sigma}\cap \S\vert \geq (q^3-2q^2+6)/3$ for some $i\neq j\in \lbrace 1,\ldots,3\rbrace$. Denote the number of distinct 3-dimensional spaces $Q_i^{\sigma}\cap Q_j^{\sigma}$ such that $\vert Q_i^{\sigma}\cap Q_j^{\sigma}\cap \S \vert$ is at least $(q^3-2q^2+6)/3$ by $x$, and suppose that two such distinct $3$-spaces share less than $q^3/105$ common points with $\S$. Then
\begin{eqnarray}\sum_{i=0}^{x-1} ((q^3-2q^2+6)/3-i(q^3/105))\leq 2(q^3+q^2).\label{ve}\end{eqnarray}

It is easy to see that there are at least seven $3$-spaces $Q_i^{\sigma}\cap Q_j^{\sigma}\cap \S$ such that $\vert Q_i^{\sigma}\cap Q_j^{\sigma}\cap \S\vert$ is at least $(q^3-2q^2+6)/3$. Filling in $x=7$ in inequality (\ref{ve}) yields a contradiction since $q\geq 50$. This implies that there is a plane containing at least $q^3/105>(17q^2-270q+10)/2$ points of $\S$ if $q>879$.\end{proof}

\begin{theorem} \label{herm} Let $c$ be a codeword of $\Ha$, $q>893$, of weight at most $2(q^3+q^2)$. Then one of the following cases holds:
\begin{itemize}
\item $wt(c)=2(q^3+1)$ and $\S$ is the union of 2 Hermitian curves $\Gamma$ and $\Gamma'$ of $\mathcal{H}(5,q^2)$, in polar planes $\pi$ and $\pi^{\sigma}$, where $\pi\not\subseteq \mathcal{H}(5,q^2)$. In the codeword $c$, all points of $\Gamma$ have symbol $\alpha$ and all points of $\Gamma'$ have symbol $-\alpha$.

\item $wt(c)=2(q^3+q^2)$ and $\S$ is the union of two cones $\Delta$ and $\Delta'$ of $\mathcal{H}(5,q^2)$, both with vertex $P$ and base a Baer subline, where $\Delta\subset \pi$ for some plane $\pi\not\subseteq \mathcal{H}(5,q^2)$ and $\Delta'\subset \pi^{\sigma}$. In the codeword $c$, all points of $\Delta$ have symbol $\alpha$ and all points of $\Delta'$ have symbol $-\alpha$, except for the point $P$ which has symbol zero.
\end{itemize}
\end{theorem}
\begin{proof} According to Lemma \ref{?}, we find a plane $\pi$ with more than $(17q^2-270q+10)/2$ points of $\S$, and Lemma \ref{vlakvanH} shows that $\pi^{\sigma}\neq\pi$. 

{\bf Case 1: $\pi\cap \pi^{\sigma}=\emptyset$.}

Lemma \ref{le1} shows that there are at least $q^3-5q/2-5/2$ points $P_i$ of $\S$ contained in this plane $\pi$, and at least $q^3-5q/2-5/2 $ points $Q_j$ of $\S$ contained in $\pi^{\sigma}$. Hence, there are at most $2q^2+5q+5$ points in $\S\setminus (\pi \cup \pi^{\sigma})$. These points are collinear with at most $(2q^2+5q+5)(q+1)$ of the points $P_i$ of $\S$ in $\pi$. This implies that some point $P_1$ of the points $P_i$ is incident with at most $(2q^2+5q+5)(q+1)/(q^3-5q/2-5/2)<2$ points of $\S\setminus (\pi\cup\pi^{\sigma})$. Hence, there are at most $2(q+1)$ planes of $\mathcal{H}(5,q^2)$ through $P_1$ and one of these 2 points, containing a point $Q_j$ of $\S$ in $\pi^{\sigma}$. If $P_1$ has symbol $\alpha$, we conclude that there are at least $|\pi^{\sigma}\cap \S|-2(q+1)\geq q^3-9q/2-9/2$ points $Q_j$ with symbol $-\alpha$. The same calculation for a point $Q_j$ shows that at least $q^3-9q/2-9/2$ of the points $P_i$ have symbol $\alpha$.

{\bf Case 2: $\pi\cap \pi^{\sigma}=\{P\}$.}

Lemma \ref{le2} shows that there are at least  $q^3-6q^2$ points $P_i$ of $\S$ contained in this plane $\pi$, and at least $q^3-6q^2$ points $Q_j$ of $\S$ contained in $\pi^{\sigma}$. Hence, there are at most $14q^2+1$ points $R_k$ of $\S$, not in $P^{\sigma}$. A point $R_1$ is collinear with at most $q+1$ of the points $P_i$, hence, there is a point of $\S\cap \pi$, say $P_1$, collinear with at most $(14q^2+1)(q+1)/(q^3-6q^2-1)<15$ of the points $R_k$. This implies that at least $q^3-6q^2-1-(14q^2+1)-14(q+1)=q^3-20q^2-14q-16$ of the points $Q_k$, say $Q_1,\ldots,Q_x$, $x\geq q^3-20q^2-14q-16$, are collinear with $P_1$, where $P_1Q_i,i=1,\ldots, x$, is a line of $\mathcal{H}(5,q^2)$, not containing any other point of $\S$, and where the $q$ planes of $\mathcal{H}(5,q^2)$ through $P_1Q_i$, but not through $P$, do not contain any other point of $\S$. So if $P_1$ has symbol $\alpha$, then at least $q^3-20q^2-14q-16$ of the points $Q_j$ have symbol $-\alpha$. The same calculation for a point $Q_j$with symbol $-\alpha$ shows that at least $q^3-20q^2-14q-16$ of the points $P_i$ have symbol $\alpha$.\\

We now prove that the codeword $c$ is as described in the statement of Theorem \ref{herm}. Let $c'$ be the codeword of $\mathrm{C}(\mathcal{H}(5,q^2))^{\bot}$ such that all points of $\mathcal{H}(5,q^2)\cap \pi$ have symbol $\alpha$, and all points of $\mathcal{H}(5,q^2)\cap \pi^{\sigma}$ symbol $-\alpha$. If $\mathcal{H}(5,q^2)\cap \pi$ is a cone with vertex $P$, give the point $P$ symbol zero in $c'$.
 
 Then $c-c'$ is a codeword of the linear code $\Ha$, it has weight at most $2(q^3+q^2)+2(q^3+q^2)-2(2q^3-40q^2-28q-32)=84q^2+56q+64$. Since the minimum weight of $\Ha$ is at least $q^3+2$, $c=c'$, and the theorem is proven.
\end{proof}

\section{The dual code of $\mathcal{Q}^-(5,q)$ and $\mathcal{H}(4,q^2)$}

In this section, we give a geometric description of some small weight codewords   in the codes $\mathrm{C}_1(\mathcal{Q}^-(5,q))^{\bot}=\mathrm{C}(\mathcal{Q}^-(5,q))^{\bot}$ and $\mathrm{C}_1(\mathcal{H}(4,q^2))^{\bot}=\mathrm{C}(\mathcal{H}(4,q^2))^{\bot}$.

Let $c$ be a codeword of the $p$-ary code $\mathrm{C}(\mathcal{Q}^-(5,q))^{\bot}$, $q=p^h$, $p$ prime, $h\geq 1$, and let $\S$ be $supp(c)$. Then $\S$ defines a set of points of $\mathcal{Q}^-(5,q)$ such that every line of $\mathcal{Q}^-(5,q)$ contains $0$ or at least $2$ points of $\S$, and such that the sum of the symbols of the points on a line equals zero. 

The minimum weight of $\mathrm{C}(\mathcal{Q}^-(5,q))^{\bot}$ is not known. Let us first derive a trivial lower bound on this minimum weight.
Let $P$ be a point of $\S$. Since $P^{\sigma}$ is a cone with vertex $P$ and base $\mathcal{Q}^-(3,q)$, it is clear that $\vert P^{\sigma}\cap \S\vert \geq q^2+2$. Hence, the minimum weight of $\mathrm{C}(\mathcal{Q}^-(5,q))^{\bot}$ is at least $q^2+2$. 

However, this trivial lower bound is weak: the {\em bit-oriented bound}, the {\em parity-oriented bound}, and the
{\em tree bound}, developed by Tanner (\cite{T1}, \cite{T2}, \cite{T3}), yield that the minimum weight of $\mathrm{C}(\mathcal{Q}^-(5,q))^{\bot}$ is at least $q^3+q+2$. In the following proposition, we construct a codeword of weight less than twice this lower bound.

\begin{proposition}\label{prop}
Let $d$ be the minimum weight of the $p$-ary code $\mathrm{C}(\mathcal{Q}^-(5,q))^{\bot}$, $q=p^h$, $p$ prime, $h\geq 1$, then
$$q^3+q+2\leq d\leq 2(q^3-q^2+q).$$
\end{proposition}
\begin{proof}
The left hand side of this inequality follows from \cite{T1}, \cite{T2}, \cite{T3}. 

Let $P_1$ and $P_2$ be non-collinear points of $\mathcal{Q}^-(5,q)$, with $P_1^{\sigma}\cap P_2^{\sigma}$  an elliptic quadric $\mathcal{Q}=\mathcal{Q}^-(3,q)$. Let $c$ be the vector where all points of the cone $P_1\mathcal{Q}\setminus \mathcal{Q}$ have symbol $\alpha$, all points of $P_2\mathcal{Q}\setminus \mathcal{Q}$ have symbol $-\alpha$, and all other points have symbol zero. 

Let $P$ be a point of $\S=supp(c)$; assume without loss of generality that $P$ is collinear with $P_1$ and let $L$ be a line of $\mathcal{Q}^-(5,q)$ through $P$. There are two possibilities. Either $L$ is a line of $P_1\mathcal{Q}$, and then the sum of the symbols of the points on $L$ is $q\cdot\alpha=0\pmod{p}$, or $L$ is a line not in $P_1\mathcal{Q}$. There are $q^2+1$ lines through $P$, and every one of the $q^2+1$ lines through $P_2$ has exactly one point collinear with $P$ since $\mathcal{Q}^-(5,q)$ is a generalised quadrangle. Hence, if $L$ is not a line of $P_1\mathcal{Q}$, it intersects exactly one of the lines of $P_2\mathcal{Q}$ in a point with symbol $-\alpha$. 
This implies that for every line through $P$, the sum of the symbols of the points on this line equals zero, hence $c$ is a codeword of $\mathrm{C}(\mathcal{Q}^-(5,q))^{\bot}$, and it has weight $2(q^2+1)(q-1)+2$.
\end{proof}

For the code $\mathrm{C}(\mathcal{H}(4,q^2))^{\bot}$, we can proceed in the same way. Again, by \cite{T1}, \cite{T2}, and \cite{T3}, we find a lower bound  $q^5-q^4+q^3+q^2+2$ on the minimum weight, strongly improving the trivial lower bound $q^3+2$.

In the following proposition, we construct a small weight codeword in the code $\mathrm{C}(\mathcal{H}(4,q^2))^{\bot}$.

\begin{proposition}\label{prop2}
Let $d$ be the minimum weight of the $p$-ary code $\mathrm{C}(\mathcal{H}(4,q^2))^{\bot}$, $q=p^h$, $p$ prime, $h\geq 1$, then
$$q^5-q^4+q^3+q^2+2\leq d\leq 2(q^5-q^3+q^2).$$
\end{proposition}
\begin{proof}
The left hand side of this inequality follows from \cite{T1}, \cite{T2}, \cite{T3}. 

Let $P_1$ and $P_2$ be non-collinear points of $\mathcal{H}(4,q^2)$, with $P_1^{\sigma}\cap P_2^{\sigma}$ a Hermitian curve $\mathcal{H}=\mathcal{H}(2,q^2)$. Let $c$ be the vector where all points of the cone $P_1\mathcal{H}\setminus \mathcal{H}$ have symbol $\alpha$, all points of $P_2\mathcal{H}\setminus \mathcal{H}$ have symbol $-\alpha$, and all other points of $\mathcal{H}(4,q^2)$ have symbol zero.
Using the same proof as for Proposition \ref{prop}, $c$ is a codeword, and it has weight $2(q^2-1)(q^3+1)+2=2(q^5-q^3+q^2)$.
\end{proof}

\section{The dual code of $\mathcal{Q}^+(2n+1,q)$}
\subsection{An upper bound on the minimum weight}


In Proposition \ref{minimum}, we have derived a lower bound on the minimum weight of $\mathrm{C}_k(\mathcal{Q}^+(2n+1,q))^{\bot}$. In the next proposition and its corollary, we derive an upper bound on the minimum weight of $\mathrm{C}_n(\mathcal{Q}^+(2n+1,q))^\bot$ by constructing a codeword of small weight. Let $\sigma$ be the polarity defined by $\mathcal{Q}^+(2n+1,q)$.

\begin{proposition}\label{klein} Let $\pi$ be an $n$-dimensional space not contained in $\mathcal{Q}^+(2n+1,q)$, $q=p^h$, $p$ prime, $h\geq 1$. Let $T$ be the intersection of $\pi$ and $\pi^{\sigma}$. Let $c$ be the vector where all points of $(\pi\cap \mathcal{Q}^+(2n+1,q))\setminus T$ have symbol $\alpha$, and all points of $(\pi^{\sigma}\cap \mathcal{Q}^+(2n+1,q))\setminus T$ have symbol $-\alpha$. Then $c$ is a codeword of $\mathrm{C}_n(\mathcal{Q}^+(2n+1,q))^{\bot}$.
\end{proposition}

\begin{proof} Let $P$ be a point of $\S=supp(c)$, suppose without loss of generality that it has symbol $\alpha$. Let $\mu$ be a generator of $\mathcal{Q}^+(2n+1,q)$ through $P$. Suppose first that $\mu$ does not contain a point of $\pi\cap \pi^{\sigma}$. Then $\pi\cap\mu$ contains $1\pmod{p}$ points with symbol $\alpha$. The $n$-spaces $\mu=\mu^{\sigma}$ and $\pi^{\sigma}$ are contained in the $2n$-dimensional space $P^{\sigma}$, so they have a non-empty intersection. Hence, $\mu\cap \pi^{\sigma}$ contains $1\pmod{p}$ points with symbol $-\alpha$. 
Suppose that $\mu$ contains points of $\pi\cap\pi^{\sigma}$. Then $\pi\cap\mu$ contains $0\pmod{p}$ points with symbol $\alpha$. The $n$-spaces $\mu=\mu^{\sigma}$ and $\pi^{\sigma}$ are contained in the $2n$-dimensional space $P^{\sigma}$, so they intersect in a subspace of $\mathcal{Q}^+(2n+1,q)$ of dimension at least zero. Since $\mu$ contains points of $\pi\cap\pi^{\sigma}$, $\mu\cap \pi^{\sigma}$ contains $0\pmod{p}$ points with symbol $-\alpha$. 
This implies that for every generator through $P$, the sum of the symbols of the points in this generator equals zero, hence $c$ is a codeword of $\mathrm{C}_n(\mathcal{Q}^+(2n+1,q))^{\bot}$.
\end{proof}
\begin{corollary} 
Let $d$ be the minimum weight of the $p$-ary code $\mathrm{C}_n(\mathcal{Q}^+(2n+1,q))^{\bot}$, $q=p^h$, $p$ prime, $h\geq 1$, then $d\leq 2(q^n-1)/(q-1)$ if $n$ is even, and $d\leq 2(q^n-1)/(q-1)-2q^{(n-1)/2}$ if $n$ is odd.
\end{corollary}
\begin{proof} Let $n$ be even. The codeword, constructed in Proposition \ref{klein}, starting with an $n$-space intersecting $\mathcal{Q}^+(2n+1,q)$ in a parabolic quadric $\mathcal{Q}(n,q)$, has weight $2(q^n-1)/(q-1)$. Let $n$ be odd. If an $n$-space $\pi$ intersects $\mathcal{Q}^+(2n+1,q)$ in an elliptic quadric $\mathcal{Q}^-(n,q)$, then $\pi^{\sigma}$ intersects $\mathcal{Q}^+(2n+1,q)$ also in an elliptic quadric. Hence, the codeword constructed in Proposition \ref{klein} has weight $2(q^n-1)/(q-1)-2q^{(n-1)/2}$.
\end{proof}

\subsection{Large weight codewords for $q$ even}
Large weight codewords of  $\mathrm{C}_k(\mathcal{Q}^+(2n+1,q))^{\bot}$, $q$ even, correspond to small blocking sets with respect to $k$-subspaces of $\mathcal{Q}^+(2n+1,q)$, hence we start by considering minimal blocking sets of $\mathcal{Q}^{+}(2n+1,q)$ with respect to 
$k$--subspaces.

\textbf{The case $k=n$.} An {\em ovoid} of $\mathcal{Q}^+(2n+1,q)$ is a set $\mathcal{O}$ of points such that every $n$-space, contained in $\mathcal{Q}^+(2n+1,q)$, contains exactly one point of $\mathcal{O}$. Obviously, ovoids of $\mathcal{Q}^+(2n+1,q)$ are the smallest possible blocking sets with respect to $n$-spaces. The problem of the existence of ovoids in
$\mathcal{Q}^{+}(2n+1,q)$, $n>2$, $q>3$, is still an open problem (for more
details see \cite{Thas}), hence we will not treat this case.

\textbf{The case $k=1$.} The problem of determining the smallest
blocking sets with respect to the lines of $Q^+(2n+1,q)$ is completely solved in
the following result by Metsch.

\begin{result}\label{klaus1} \cite{Metsch}
Let $B$ be a minimal blocking set of $\mathcal{Q}^{+}(2n+1,q)$, $n\geq 3$, with respect to lines. If  $|B|\leq 1+q|\mathcal{Q}^+(2n-1,q)|$, then $B=(T\setminus T^{\sigma})\cap Q$ for some hyperplane $Q$ of $\PG(2n+1,q)$.

\end{result}

\begin{theorem}\label{eerste}
The maximum weight of $\mathrm{C}_1(\mathcal{Q}^{+}(2n+1,q))^{\bot}$, $q$ even, $n\geq 3$, is $(q^{n}+1)q^{n}$, the second
largest weight is $q^{2n}$, and the codewords of weight $(q^{n}+1)q^{n}$ and $q^{2n}$ are defined by the complement of a hyperplane with respect to $\mathcal{Q}^+(2n+1,q)$.
\end{theorem}
\begin{proof} 
Let $c$ be a codeword of $\mathrm{C}_1(\mathcal{Q}^+(2n+1,q))^{\bot}$ with $wt(c)\geq q^{2n}$. Then $\B$ is a blocking set with respect to lines of $\mathcal{Q}^+(2n+1,q)$ of size at most $\theta_{2n-1}+q^{n}$. Result \ref{klaus1} shows that $\B$ contains either a parabolic quadric $\mathcal{Q}(2n,q)$ or a cone over a hyperbolic quadric $\mathcal{Q}^+(2n-1,q)$ minus its vertex.

Suppose first that the blocking set $\B$ contains a parabolic quadric $\mathcal{Q}=\mathcal{Q}(2n,q)$. A line $L$ of $\mathcal{Q}^+(2n+1,q)$ is either contained in $\mathcal{Q}$ or it
intersects $\mathcal{Q}$ in one point. Hence, by conditions $(*)$ and $(**)$,
the complement of $\mathcal{Q}$ defines a codeword $c'$ of $\mathrm{C}_1(\mathcal{Q}^+(2n+1,q))^{\bot}$. But then $wt(c+c')=wt(c)+wt(c')-2wt(c\cap c')\leq   q^n$, which is smaller than the minimum weight of $\mathrm{C}_1(\mathcal{Q}^+(2n+1,q))^{\bot}$ (see Proposition \ref{minimum}). Hence, in this case, $c=c'$, $wt(c)=q^{2n}+q^n$, and $\S$ is the complement of a parabolic quadric $\mathcal{Q}(2n,q)$.

Suppose that $\B$ contains a set $P^{\sigma} \setminus \{P\}$, with $P \in \mathcal{Q}^{+}(2n+1,q)$, then the complement of $c$ has size $\theta_{2n-1}+q^{n}-1$ or $\theta_{2n-1}+q^{n}$. The set $P^{\sigma}$ is such that a line is either contained in $P^{\sigma}$ or contains one point of $P^{\sigma}$. Hence, by conditions $(*)$ and $(**)$, the complement of $P^{\sigma}$ defines a codeword $c'$ of $\mathrm{C}_1(\mathcal{Q}^+(2n+1,q))^{\bot}$. But then $wt(c+c')\leq 1$, which implies that $c=c'$, $wt(c)=q^{2n}$, and $c$ corresponds to the complement of a set $P^{\sigma}$ with $P\in \mathcal{Q}^+(2n+1,q).$ 
\end{proof}

\begin{corollary} The largest sets of even type in $\mathcal{Q}^+(2n+1,q)$, $q$ even, $n\geq 3$, have size $q^{2n}+q^n$, and they correspond to the complement of a parabolic quadric $\mathcal{Q}(2n,q)$ of $\mathcal{Q}^+(2n+1,q)$, there are no sets of even type in $\mathcal{Q}^+(2n+1,q)$, $q$ even, $n\geq 3$, with weight in $[q^{2n}+1,q^{2n}+q^n[$, and a set of even type in $\mathcal{Q}^+(2n+1,q)$, $q$ even, $n\geq 3$, of size $q^{2n}$ corresponds to the complement of a cone over a hyperbolic quadric $\mathcal{Q}^+(2n-1,q)$.
\end{corollary}


\textbf{The case $2\leq k \leq n-1$.} Let $S_{k}\mathcal{Q}$ be a
cone with a $k$-dimensional vertex $S_k$ over a non--singular quadric $\mathcal{Q}$. A {\em truncated} cone $S_{k}\mathcal{Q}$ is the set $S_{k}\mathcal{Q}\setminus S_k$.
\begin{result}(\cite{Metsch2})\label{M}
Let $B$ be a blocking set of the quadric $\mathcal{Q}^{+}(2n+1,q)$
with respect to $k$--subspaces, $2 \leq k \leq n-1$. Then $|B|\geq
\frac{q^{n-k+1}-1}{q-1}(q^{n}+q^{k-2})$. If $|B|<
\frac{q^{n-k+1}-1}{q-1}(q^{n}+q^{k-2}+1)$, then $B$ contains the
truncated cone
$S_{k-3}\mathcal{Q}^{-}(2n+3-2k,q)$.
\end{result}

Hence we can prove the following theorem.

\begin{theorem}\label{tweede}
The maximum weight of $\mathrm{C}_2(\mathcal{Q}^{+}(2n+1,q))^{\bot}$, $q$ even, is $(q^{n}+1)(q^{n}+q^{n-1})$
and a codeword of maximum weight corresponds to the complement of an elliptic quadric $\mathcal{Q}^{-}(2n-1,q)$
in $\mathcal{Q}^{+}(2n+1,q)$.
\end{theorem}
\begin{proof}
By Result \ref{M}, the smallest minimal blocking set with respect to the
planes of $Q^+(2n+1,q)$ is an elliptic quadric $\mathcal{Q}^{-}(2n-1,q)$. Every plane of $\mathcal{Q}^+(2n+1,q)$  intersects $\mathcal{Q}^{-}(2n-1,q)$ in an odd number of points, hence by
conditions $(*)$ and $(**)$, the complement of such an elliptic quadric
$\mathcal{Q}^{-}(2n-1,q)$ defines a codeword of the largest weight.
\end{proof}

\begin{theorem}\label{derde}
The maximum weight of $\mathrm{C}_k(\mathcal{Q}^{+}(2n+1,q))^{\bot}$, $q$ even, with $3 \leq k <(n+3)/2,$ is $q^{n}(q^{n}+q^{n-1} +\cdots+ q^{n-k+1})+q^{n}+q^{n-1}$ and
a codeword of maximum weight corresponds to the complement of a cone
$S_{k-3}\mathcal{Q}^{-}(2n+3-2k,q)$ in $\mathcal{Q}^{+}(2n+1,q)$.
\end{theorem}
\begin{proof}
Denote the size of the complement of a cone $S_{k-3}\mathcal{Q}^{-}(2n+3-2k,q)$ by $s$. Let $c$ be a codeword of $\mathrm{C}_k(\mathcal{Q}^{+}(2n+1,q))^{\bot}$ with $wt(c)\geq s$. The complement of $c$ corresponds to a blocking set of size at most $\vert S_{k-3}\mathcal{Q}^{-}(2n+3-2k,q)\vert$. Hence, by Result \ref{M}, the complement of $c$ consists of the points of a truncated cone $\mathcal{C}=S_{k-3}\mathcal{Q}^{-}(2n+3-2k,q)$ and some set of other points, which we will denote by $T$. Since $wt(c)\geq s$, $\vert T \vert \leq \vert S_{k-3}\vert=\theta_{k-3}$. Let $c'$ be the codeword corresponding to the complement of the cone $\mathcal{C'}$, obtained by adding the vertex $S_{k-3}$ to the truncated cone $\mathcal{C}$. Since $\mathrm{C}_k(\mathcal{Q}^{+}(2n+1,q))^{\bot}$ is a linear code, the vector $c+c'$ is in $\mathrm{C}_k(\mathcal{Q}^{+}(2n+1,q))^{\bot}$. Moreover, it has weight
$$wt(c+c')= wt(c)+wt(c')-2wt(c\cap c')
\leq 2\vert S_{k-3}\vert.$$
But $2\vert S_{k-3}\vert$ is smaller than $1+\frac{q^{n}-1}{q^{k}-1}(q^{n-1}+1)$ which is the lower bound on the minimum weight by Proposition \ref{minimum}. This implies that $c=c'$. Hence, the maximum weight of $\mathrm{C}_k(\mathcal{Q}^{+}(2n+1,q))^{\bot}$ is  $q^{n}(q^{n}+q^{n-1}+ \cdots +q^{n-k+1})+q^{n}+q^{n-1}$ and corresponds to the complement of a cone $S_{k-3}\mathcal{Q}^{-}(2n+3-2k,q)$ in $\mathcal{Q}^{+}(2n+1,q)$.
\end{proof}

\section{The dual code of $\mathcal{Q}(2n,q)$}


If we consider $\mathcal{Q}(2n,q)$ to be embedded in
$\mathcal{Q}^{+}(2n+1,q)$, then every blocking set of
$\mathcal{Q}(2n,q)$ with respect to subspaces of dimension $k$ of
$\mathcal{Q}(2n,q)$ is a blocking set of $\mathcal{Q}^{+}(2n+1,q)$
with respect to subspaces of dimension $k+1$ of
$\mathcal{Q}^{+}(2n+1,q)$. So, for $1\leq k\leq n-2$, a blocking set $B$ of smallest size
consists of the non--singular points of a quadric of type
$S_{k-2}\mathcal{Q}^{-}(2n+1-2k,q)$ (Result \ref{M}; for more details, see
\cite{Metsch2}). From this consideration and by similar arguments of
the previous section, we get the following result:

\begin{theorem}\label{vierde}
The maximum weight for $\mathrm{C}_k(\mathcal{Q}(2n,q))^{\bot}$, $q$ even, $1\leq k< (n+1)/2$, is $q^{n}(q^{n-1}+q^{n-2}+ \cdots+ q^{n-k})+q^{n-1}$ and a codeword of maximum weight corresponds to the complement of a cone
$S_{k-2}\mathcal{Q}^{-}(2n+1-2k,q)$ in $\mathcal{Q}(2n,q)$.
\end{theorem}
\begin{corollary} The largest sets of even type in $\mathcal{Q}(2n,q)$, $q$ even, have size $q^{2n-1}+q^{n-1}$, and correspond to the complement of an elliptic quadric $\mathcal{Q}^-(2n-1,q)$.
\end{corollary}

\section{The dual code of $\mathcal{Q}^{-}(2n+1,q)$}
%



Now consider $\mathcal{Q}^{-}(2n+1,q)$ embedded in
$\mathcal{Q}^{+}(2n+3,q)$, then every blocking set of
$\mathcal{Q}^{-}(2n+1,q)$ with respect to subspaces of dimension $k$ of
$\mathcal{Q}^-(2n+1,q)$ is a blocking set of $\mathcal{Q}^{+}(2n+3,q)$
with respect to subspaces of dimension $k+2$ of
$\mathcal{Q}^{+}(2n+3,q)$. So, for $1\leq k\leq n-2$, a blocking set $B$ of smallest size
consists of the non--singular points of a quadric of type
$S_{k-1}\mathcal{Q}^{-}(2n+1-2k,q)$ (\cite{Metsch2}).

Hence, for the codewords of large weight we have:

\begin{theorem}\label{vijfde}
The maximum weight for
$\mathrm{C}_{k}(\mathcal{Q}^{-}(2n+1,q))^{\bot}$, $q$ even, $1\leq k< (n+1)/2$, is $q^{2n-k+1}\theta_{k-1}$ and a codeword of maximum weight corresponds to the
complement of a cone $S_{k-1}\mathcal{Q}^{-}(2n+1-2k,q)$ in 
$\mathcal{Q}^{-}(2n+1,q)$.
\end{theorem}

\begin{corollary} The largest sets of even type in $\mathcal{Q}^{-}(2n+1,q)$, $q$ even, have size $q^{2n}$ and correspond to the complement of a cone over an elliptic quadric $\mathcal{Q}^{-}(2n-1,q)$.

\end{corollary}

\section{The dual code of $\mathcal{H}(n,q^{2})$}
\subsection{An upper bound on the minimum weight for odd $n$}

\begin{proposition}\label{ext} Let $\pi$ be an $(n-1)/2$-dimensional space, not contained in $\mathcal{H}(n,q^2)$, $n$ odd, $q=p^h$, $p$ prime, $h\geq 1$. Let $T$ be the intersection of $\pi$ and $\pi^{\sigma}$, where $\sigma$ is the Hermitian polarity defined by $\mathcal{H}(n,q^2)$. Let $c$ be the vector where all points of $(\pi\cap \mathcal{H}(n,q^2))\setminus T$ have symbol $\alpha$, and all points of $(\pi^{\sigma}\cap \mathcal{H}(n,q^2))\setminus T$ have symbol $-\alpha$. Then $c$ is a codeword of $\mathrm{C}_{(n-1)/2}(\mathcal{H}(n,q^2))^{\bot}$.
\end{proposition}
\begin{proof}
The same arguments as in the proof of Proposition \ref{klein} prove this statement.
\end{proof}
Note that Example \ref{vb} was constructed in this way.





%

\subsection{Large weight codewords for $q$ even}
We summarize the results of Metsch on blocking sets of Hermitian varieties in the following result. Let $S_{i}\mathcal{H}(n,q^{2})$ be a cone with $i$--dimensional
vertex $S_{i}$ over a Hermitian variety $\mathcal{H}(n,q^{2})$.

\begin{result}(\cite{Metsch3})\label{M2}
Let $B$ be a minimal blocking set with respect to the $k$-subspaces of a Hermitian variety $\mathcal{H}(n,q^{2})$  and let
$k\leq \frac{n-3}{2}$. If $|B|\leq
q^{2(n-k-7)}|\mathcal{H}(7,q^{2})|+q^{3}$, then $B$ is of type
$S_{i}\mathcal{H}(n-k-1-i,q^{2}) \setminus S_{i}$. Thus, if $n$ is
even, then the smallest blocking set is $S_{k-1}\mathcal{H}(n-2k,q^{2})
\setminus S_{k-1}$; and when $n$ is odd, the smallest blocking set is
$S_{k-2}\mathcal{H}(n-2k+1,q^{2}) \setminus S_{k-2}$.
\end{result}

This result enables us to find the large weight codewords of the code $\mathrm{C}_{k}(\mathcal{H}(n,q^2))^{\bot}$,
$1\leq k \leq \frac{n-3}{2}$, $q$ even.

\begin{theorem}\label{zesde} If $n$ is even, the maximum weight of $\mathrm{C}_{k}(\mathcal{H}(n,q^2))^{\bot}$, $q$ even, $1\leq k\leq(n-3)/2$, is $q^{2n-2k+1}\frac{q^{2k}-1}{q^{2}-1}$
and a codeword of maximum weight corresponds to the complement of a cone
$S_{k-1}\mathcal{H}(n-2k,q^{2})$.
If $n$ is odd, then the maximum weight of $\mathrm{C}_{k}(\mathcal{H}(n,q^2))^{\bot}$ is $q^{2n-2k+1}\frac{q^{2k}-1}{q^{2}-1}+q^{n-1}$ and a codeword of maximum weight corresponds to
the complement of a cone $S_{k-2}\mathcal{H}(n-2k+1,q^{2})$.
\end{theorem}
\begin{proof}
Let $c$ be a codeword of $\mathrm{C}_{k}(\mathcal{H}(n,q^2))^{\bot}$ with $wt(c)$ at least the size of the complement of a cone $S_{k-1}\mathcal{H}(n-2k,q^{2})$ if $n$ is even, or $S_{k-2}\mathcal{H}(n-2k+1,q^{2})$ if $n$ is odd. The complement of $c$ corresponds to a blocking set of size at most $\vert S_{k-1}\mathcal{H}(n-2k,q^{2})\vert$ if $n$ is even, or $\vert S_{k-2}\mathcal{H}(n-2k+1,q^{2})\vert$ if $n$ is odd. Hence, by Result \ref{M2}, the complement of $c$ consists of the points of a truncated cone and some set of other points. One can check that $S_{k-1}\mathcal{H}(n-2k,q^{2})\setminus S_{k-1}$ is the only truncated cone with size less than $\vert S_{k-1}\mathcal{H}(n-2k,q^{2})\vert$ if $n$ is even, and $S_{k-2}\mathcal{H}(n-2k+1,q^{2})$ is the only truncated cone with size less than $\vert S_{k-2}\mathcal{H}(n-2k+1,q^{2})\vert$ if $n$ is odd. Hence, the complement of $c$ consists of a truncated cone $\mathcal{C}_1=S_{k-1}\mathcal{H}(n-2k,q^{2})$ if $n$ is even, or $\mathcal{C}_2=S_{k-2}\mathcal{H}(n-2k+1,q^{2})$ if $n$ is odd, and some set of other points, which we will denote by $T$. Since $wt(c)$ is at least the size of the complement of a cone $S_{k-1}\mathcal{H}(n-2k,q^{2})$ if $n$ is even, and $S_{k-2}\mathcal{H}(n-2k+1,q^{2})$ if $n$ is odd, $|T \vert \leq \vert S_{k-1}\vert=\theta_{k-1}$ if $n$ is even, and $\vert T \vert \leq \vert S_{k-2}\vert=\theta_{k-2}$ if $n$ is odd. Let $c'$ be the codeword of $\mathrm{C}_{k}(\mathcal{H}(n,q^2))^{\bot}$ corresponding to the complement of the cone $\mathcal{C'}$, obtained by adding the vertex to the truncated cone $\mathcal{C}_1$ if $n$ is even, or $\mathcal{C}_2$ if $n$ is odd. Since $\mathrm{C}_{k}(\mathcal{H}(n,q^2))^{\bot}$ is a linear code, the vector $c+c'$ is in $\mathrm{C}_{k}(\mathcal{H}(n,q^2))^{\bot}$. Moreover, it has weight
$$wt(c+c')= wt(c)+wt(c')-2wt(c\cap c')
\leq 2\vert S_{k-i}\vert,$$
where $i=1$ if $n$ is even, and $i=2$ if $n$ is odd. But $2\vert S_{k-i}\vert$ is smaller than $1+\frac{(q^{n-1}-(-1)^{n-1})(q^{n-2}-(-1)^{n-2})}{q^{2k}-1}$ if $k\leq(n-3)/2$, which is the lower bound on the minimum weight by Proposition \ref{minimum}. This implies that $c=c'$. Hence, the maximum weight of $\mathrm{C}_{k}(\mathcal{H}(n,q^2))^{\bot}$ for $n$ even is $q^{2n-2k+1}\frac{q^{2k}-1}{q^{2}-1}$
and it corresponds to the complement of a cone
$S_{k-1}\mathcal{H}(n-2k,q^{2})$.
For $n$ odd, the maximum weight of $\mathrm{C}_{k}(\mathcal{H}(n,q^2))^{\bot}$ is $q^{2n-2k+1}\frac{q^{2k}-1}{q^{2}-1}+q^{n-1}$ and it corresponds to the complement of a cone $S_{k-2}\mathcal{H}(n-2k+1,q^{2})$.
\end{proof}

\begin{corollary} The largest sets of even type of $\mathcal{H}(n,q^2)$, $q$ even,  have size $q^{2n-1}$ and correspond to the complement of a cone over a Hermitian variety $\mathcal{H}(n-2,q^2)$ if $n$ is even, and to the complement of a Hermitian variety $\mathcal{H}(n-1,q^2)$ if $n$ is odd.
\end{corollary}


\textbf{Acknowledgement:} Part of this research was done when the first author was visiting the Incidence Geometry research group of the Department of Pure Mathematics and Computer Algebra at Ghent University. The first author wishes to thank the members of this research group for their hospitality and the financial support offered to her.

The authors also wish to thank the referees for their many suggestions for improving the first version of this article.

Address of the authors: \\
\\
Valentina Pepe:\\
Dipartimento di Matematica e Applicazioni "R. Caccioppoli"\\
Universit\`a degli Studi di Napoli Federico II\\
Via Cintia - Monte S. Angelo\\
80126 Napoli (Italia)\\
valepepe@unina.it\\
\\
Leo Storme and Geertrui Van de Voorde:\\
Department of pure mathematics and computer algebra\\
Ghent University\\
Krijgslaan 281-S22\\
9000 Ghent  (Belgium)\\
$\{$ls,gvdvoorde$\}$@cage.ugent.be\\
http://cage.ugent.be/$\sim \{$ls, gvdvoorde$\}$\\


\begin{thebibliography}{100}
\bibitem{BK}A. Bichara and G. Korchm\'{a}ros, Note on a $(q+2)$-set
in a Galois plane of order $q$. \emph{Ann. Discrete Math.}, {\bf 14} (1982), 117--122.

\bibitem{blokhuis} A. Blokhuis, C.M. O'Keefe, S.E. Payne, L. Storme, and H. Wilbrink,  Covers of ${\rm PG}(3,q)$ and of finite generalised quadrangles.  {\em Bull. Belg. Math. Soc. Simon Stevin,} {\bf 5}  (1998),  no. 2-3, 141--157.
\bibitem{DBHS}J. De Beule, A. Hallez, and L. Storme, A non-existence
result on Cameron--Liebler line classes. \emph{J. Combin. Des.}, {\bf 16(4)}
(2007), 342--349.
\bibitem{bruen} A.A. Bruen and J.W.P. Hirschfeld, Intersections in projective space. II. Pencils of quadrics. {\em European J. Combin.} {\bf 9} (1988), 255--270.
\bibitem{ESSS}J. Eisfeld, L. Storme, T. Sz\H{o}nyi, and P. Sziklai, Covers and blocking sets of classical generalised quadrangles. \emph{Discrete
Math.,} {\bf 238} (2001), 35--51.
\bibitem{GS}P. Govaerts and L. Storme, On a particular class of
minihypers and its applications, I: The result for general $q$.
\emph{Des. Codes Cryptogr.}, {\bf 28} (2003), 51--63.
\bibitem{Hirschfeld} J.W.P. Hirschfeld, \emph{Finite Projective Spaces of Three
Dimensions}. Oxford University Press, Oxford (1986).
\bibitem{Thas} J.W.P. Hirschfeld and J.A. Thas, \emph{General Galois Geometries}. Oxford University Press,
Oxford (1991).
\bibitem{Storme} J.L. Kim, K. Mellinger, and L. Storme,
Small weight codewords in LDPC codes defined by (dual) classical
generalised quadrangles. \emph{Des. Codes Cryptogr.}, {\bf 42(1)} (2007), 73--92.
\bibitem{11}Y. Kou, S. Lin, and M.P.C. Fossorier, Low-density parity check codes based on finite geometries:
a rediscovery and new results. \emph{ IEEE Trans. Inform. Theory}, {\bf 47} (2001), 2711--2736.
\bibitem{LSV3} M. Lavrauw, L. Storme, P. Sziklai, and G. Van de Voorde,
 An empty interval in the spectrum of small weight codewords in the code of points and $k$-spaces of $\PG(n,q)$. \emph{J. Combin. Theory, Ser. A}, {\bf 116} (2009), 996--1001.
\bibitem{Liu} Z. Liu and D.A. Pados, LDPC codes from generalized polygons. \emph{IEEE Trans. Inform. Theory,} {\bf 11} (2005), 3890--3898.
\bibitem{keefe} C.M. O'Keefe, T. Penttila, and G.F. Royle, Classification of ovoids in ${\rm PG}(3,32)$.
{\em J. Geom.,} {\bf 50} (1994), 143--150. 
\bibitem{14}D.J.C. MacKay, Good error correcting codes based on very sparse matrices.\emph{ IEEE Trans.
Inform. Theory}, {\bf 45} (1999), 399--431.
\bibitem{16}D.J.C. MacKay and R.M. Neal, Near Shannon limit performance of low density parity check
codes. \emph{Electron. Lett.}, {\bf 32} (1996),1645--1646.
\bibitem{Metsch2} K. Metsch, A Bose--Burton type theorem for quadrics.
\emph{J. Combin. Des.}, {\bf 11} (5) (1999), 317--338.
\bibitem{Metsch} K. Metsch, On blocking sets of quadrics. \emph{J. Geom.} {\bf 67} (2000), 188--207.
\bibitem{Metsch3} K. Metsch, Blocking structures of hermitian varieties.
\emph{Des. Codes Cryptogr.}, {\bf 34} (2-3) (2005), 339--360.
\bibitem{payne} S.E. Payne and J.A. Thas, \emph{Finite Generalised Quadrangles}. Pitman Advanced Publishing
Program (1984).
\bibitem{val} V. Pepe, L. Storme, and G. Van de Voorde,
Small weight codewords in the LDPC codes arising from linear representations of geometries. {\em J. Comb. Designs} {\bf 17} (2009), 1--24.
\bibitem{T1} R.M. Tanner, A recursive approach to low-complexity codes. {\em IEEE Trans. Inform. Theory,} {\bf 27} (1981), 533--547.
\bibitem{T2} R.M. Tanner, Minimum-distance bounds by graph analysis. {\em IEEE Trans. Inform. Theory,} {\bf 47} (2001), 808--821.
\bibitem{ovoide} J.A. Thas, Ovoidal translation planes. {\em Arch. Math.,} {\bf 23} (1972), 110--112.
\bibitem{ovoidetits} J. Tits, Ovo\"ides \`a translations.  {\em Rend. Mat.,} {\bf 21} (1962), 37--59.
\bibitem{T3} P.O. Vontobel and R.M. Tanner, Construction of codes based on finite generalised quadrangles for iterative decoding. {\em Proceedings of 2001 IEEE international symposium on information theory, Washington, DC} (2001), 223.

\end{thebibliography}
\end{document}